\documentclass[11pt]{article}
\usepackage[english]{babel}
\usepackage[T1]{fontenc}
\usepackage[utf8]{inputenc}
\usepackage{amsbsy}
\usepackage{amsmath}
\usepackage{amssymb}
\usepackage{amsfonts}
\usepackage{amsthm}
\usepackage{bm}
\usepackage{graphicx}
\usepackage{hyperref}
\usepackage[active]{srcltx}
\usepackage{upgreek}
\usepackage{xcolor}

\newcommand{\C}{\mathbb{C}}
\newcommand{\N}{\mathbb{N}}

\newcommand{\R}{\mathbb{R}}
\renewcommand{\S}{\mathbb{S}}

\newcommand{\boC}{\mathcal{C}}

\newcommand{\boE}{\mathcal{E}}

\newcommand{\boK}{\mathcal{K}}

\newcommand{\boR}{\mathcal{R}}

\newcommand{\gE}{\mathfrak{E}}

\newcommand{\gS}{\mathfrak{S}}

\DeclareMathOperator{\arcosh}{{\rm arcosh}}
\renewcommand{\div}{\mathop{{\rm div}}\nolimits}
\DeclareMathOperator{\diag}{{\rm diag}}

\renewcommand{\Im}{\mathop{{\rm Im}}\nolimits}

\renewcommand{\Re}{\mathop{{\rm Re}}\nolimits}
\DeclareMathOperator{\sign}{{\rm sign}}

\renewcommand{\div}{\operatorname{div}}

\newtheorem{thm}{Theorem}
\newtheorem{cor}[thm]{Corollary}
\newtheorem{lem}[thm]{Lemma}
\newtheorem{prop}[thm]{Proposition}
\newtheorem*{thm*}{Theorem}
\theoremstyle{definition}
\newtheorem*{merci}{Acknowledgments}
\newtheorem{rem}[thm]{Remark}

\begin{document}

\title{The cubic Schr\"odinger regime of the Landau-Lifshitz equation with a strong easy-axis anisotropy}
\author{
\renewcommand{\thefootnote}{\arabic{footnote}}
Andr\'e de Laire\footnotemark[1]~ and Philippe Gravejat\footnotemark[2]}
\footnotetext[1]{Universit\'e de Lille, CNRS, INRIA, Laboratoire Paul Painlevé (UMR 8524), F-59000 Lille, France. E-mail: {\tt andre.de-laire@univ-lille.fr}}
\footnotetext[2]{Universit\'e de Cergy-Pontoise, Laboratoire de Math\'ematiques AGM (UMR 8088), F-95302 Cergy-Pontoise Cedex, France. E-mail: {\tt philippe.gravejat@u-cergy.fr}}
\maketitle

\begin{abstract}
We pursue our work on the asymptotic regimes of the Landau-Lifshitz equation for biaxial ferromagnets. We put the focus on the cubic Schr\"odinger equation, which is known to describe the dynamics in a regime of strong easy-axis anisotropy. In any dimension, we rigorously prove this claim for solutions with sufficient regularity. In this regime, we additionally classify the one-dimensional solitons of the Landau-Lifshitz equation and quantify their convergence towards the solitons of the one-dimensional cubic Schr\"odinger equation.

\medskip

\noindent {\it Keywords and phrases}. Landau-Lifshitz equation, nonlinear Schr\"odinger equation, asymptotic regimes.

\medskip

\noindent {\it 2010 Mathematics Subject Classification}. 35Q60; 35Q55; 37K40; 35C07; 82D40.
\end{abstract}

%%%%%%%%%%%%%%
%%%%%%%%%%%%%%
%%%%%%%%%%%%%%
\section{Introduction}
%%%%%%%%%%%%%%
%%%%%%%%%%%%%%
%%%%%%%%%%%%%%

Introduced by Landau and Lifshitz in~\cite{LandLif1}, the Landau-Lifshitz equation
\begin{equation}
\tag{LL}
\label{LL}
\partial_t m + m \times \big( \Delta m - J(m) \big) = 0,
\end{equation}
describes the macroscopical dynamics of the magnetization $m =(m_1, m_2, m_3) : \R^N \times \R \to \S^2$ in a ferromagnetic material. The possible anisotropy of the material is taken into account by the diagonal matrix $J := \diag(J_1, J_2, J_3)$, but dissipation is neglected (see e.g.~\cite{KosIvKo1}). The dynamics is Hamiltonian and the corresponding Hamiltonian is the Landau-Lifshitz energy
$$E_{\rm LL}(m) := \frac{1}{2} \int_{\R^N} \big( |\nabla m|^2 + \lambda_1 m_1^2 + \lambda_3 m_3^2 \big).$$
The characteristic numbers $\lambda_1 := J_2 - J_1$ and $\lambda_3 := J_2 - J_3$ give account of the anisotropy since they determine the preferential orientations of the magnetization with respect to the canonical axes. For biaxial ferromagnets, all the numbers $J_1$, $J_2$ and $J_3$ are different, so that $\lambda_1\neq \lambda_3$ and $\lambda_1 \lambda_3 \neq 0$. Uniaxial ferromagnets are characterized by the property that only two of the numbers $J_1$, $J_2$ and $J_3$ are equal. For instance, let us fix $J_1 = J_2$, which corresponds to $\lambda_1 = 0$ and $\lambda_3 \neq 0$, so that the material has a uniaxial anisotropy in the direction corresponding to the unit vector $e_3 = (0, 0, 1)$. In this case, the ferromagnet owns an easy-axis anisotropy along the vector $e_3$ if $\lambda_3 < 0$, while the anisotropy is easy-plane along the plane $x_3 = 0$ if $\lambda_3 > 0$. 
In the isotropic case $\lambda_1 = \lambda_3 = 0$, the Landau-Lifshitz equation reduces to the well-known Schr\"odinger map equation (see e.g.~\cite{DingWan1, SulSuBa1, ChaShUh1, BeIoKeT1} and the references therein).

In dimension one, the Landau-Lifshitz equation is completely integrable by means of the inverse scattering method (see e.g.~\cite{FaddTak0}). In this setting, it is considered as a universal model from which it is possible to derive other completely integrable equations. Sklyanin highlighted this property in~\cite{Sklyani1} by deriving two asymptotic regimes corresponding to the Sine-Gordon equation and the cubic Schr\"odinger equation.

In a previous work~\cite{deLaGra2}, we provided a rigorous derivation of the Sine-Gordon regime in any dimension $N \geq 1$. This equation appears in a regime of a biaxial material with strong easy-plane anisotropy, where the anisotropy parameters are given by
$$\lambda_1 = \sigma \varepsilon, \quad \text{and} \quad \lambda_3 = \frac{1}{\varepsilon}.$$
Here and in the sequel, $\varepsilon$ refers as usual to a small positive number, and $\sigma$ is a fixed positive constant. More precisely, we introduced a hydrodynamic formulation of the Landau-Lifshitz equation for which the solutions $m$ write as
$$m = \big( (1 - u^2)^\frac{1}{2} \sin(\phi), (1 - u^2)^\frac{1}{2} \cos(\phi), u \big),$$
and we established that the rescaled functions $(U_\varepsilon, \Phi_{\varepsilon})$ given by
$$u(x, t) = \varepsilon U_\varepsilon( \varepsilon^\frac{1}{2} \, x, t), \quad \text{and} \quad \phi(x, t) = \Phi_\varepsilon( \varepsilon^\frac{1}{2} \, x, t),$$
satisfy the Sine-Gordon system
$$\begin{cases} \partial_t U = \Delta \Phi - \frac{\sigma}{2} \sin(2 \Phi),\\
\partial_t \Phi = U, \end{cases}$$
in the limit $\varepsilon \to 0$ (under suitable smoothness assumptions on the initial datum). We refer to~\cite{deLaGra2} for more details.

We now focus on the cubic Schr\"odinger equation, which is obtained in a regime of strong easy-axis anisotropy. For this purpose, we consider a uniaxial material in the direction corresponding to the vector $e_2=(0,1,0)$ and we fix the anisotropy parameters as
$$\lambda_1 = \lambda_3 = \frac{1}{\varepsilon}.$$
For this choice, the complex map $\check{m} = m_1 + i m_3$ and the function $m_2$ corresponding to a solution $m$ to the Landau-Lifshitz equation satisfy~\footnote{Here as in the sequel, the notation $\langle z_1, z_2 \rangle_\C$ stands for the canonical real scalar product of the two complex numbers $z_1$ and $z_2$, which is given by
$$\langle z_1, z_2 \rangle_\C = \Re(z_1) \Re(z_2) + \Im(z_1) \Im(z_2) = \Re (z_1 \bar{z}_2).$$}
\begin{equation}
\label{eq:sys-LL} 
\begin{cases} i \partial_t \check{m} + m_2 \Delta \check{m}- \check{m} \Delta m_2 - \frac{1}{\varepsilon} m_2 \check{m} = 0,\\
\partial_t m_2 - \langle i \check{m}, \Delta \check{m} \rangle_{\C} = 0.
\end{cases}
\end{equation}
Let us introduce the complex-valued function $\Psi_\varepsilon$ given by
\begin{equation}
\label{def:Psi-eps}
\Psi_\varepsilon(x, t)=\varepsilon^{- \frac{1}{2}} \check{m}(x, t) e^\frac{i t}{\varepsilon}.
\end{equation}
This function is of order $1$ in the regime where the map $\check{m}$ is of order $\varepsilon^\frac{1}{2}$. When $\varepsilon$ is small enough, the function $m_2$ does not vanish in this regime, since the solution $m$ is valued into the sphere $\S^2$.
Assuming that $m_2$ is everywhere positive, it is given by the formula
$$m_2 = \big( 1 - \varepsilon |\Psi_\varepsilon|^2| \big)^\frac{1}{2},$$
and the function $\Psi_\varepsilon$ is solution to the nonlinear Schr\"odinger equation
\begin{equation}
\tag{NLS$_\varepsilon$}
\label{NLS-eps}
i \partial_t \Psi_\varepsilon + \big( 1 - \varepsilon |\Psi_\varepsilon|^2 \big)^\frac{1}{2} \Delta \Psi_\varepsilon + \frac{|\Psi_\varepsilon|^2}{1 + (1 - \varepsilon |\Psi_\varepsilon|^2)^\frac{1}{2}} \Psi_\varepsilon + \varepsilon \div \Big( \frac{\langle \Psi_\varepsilon, \nabla \Psi_\varepsilon \rangle_\C}{(1 - \varepsilon |\Psi_\varepsilon|^2)^\frac{1}{2}} \Big) \Psi_\varepsilon = 0.
\end{equation}
As $\varepsilon \to 0$, the formal limit equation is the focusing cubic Schr\"odinger equation
\begin{equation}
\tag{CS}
\label{CS}
i \partial_t \Psi + \Delta \Psi + \frac{1}{2} |\Psi|^2 \Psi = 0.
\end{equation}
Our main goal in the sequel is to justify rigorously this cubic Schr\"odinger regime of the Landau-Lifshitz equation.

We first recall some useful facts about the Cauchy problems for the Landau-Lifshitz and cubic Schr\"odinger equations. Concerning this latter equation, we refer to~\cite{Cazenav0} for an extended review of the corresponding Cauchy problem. In the sequel, our derivation of the cubic Schr\"odinger equation requires additional smoothness, so that we are mainly interested in smooth solutions for which a fixed-point argument provides the following classical result.

\begin{thm*}[\cite{Cazenav0}]
Let $k \in \N$, with $k > N/2$. Given any function $\Psi^0 \in H^k(\R^N)$, there exist a positive number $T_{\max}$ and a unique solution $\Psi \in \boC^0([0, T_{\max}), H^k(\R^N))$ to the cubic Schr\"odinger equation with initial datum $\Psi^0$, which satisfies the following statements.

\noindent $(i)$ If the maximal time of existence $T_{\max}$ is finite, then
$$\lim_{t \to T_{\max}} \| \Psi(\cdot, t) \|_{H^k} = \infty, \quad \text{and} \quad \limsup_{t \to T_{\max}} \| \Psi(\cdot, t) \|_{L^\infty} = \infty.$$

\noindent $(ii)$ The flow map $\Psi^0 \mapsto \Psi$ is well-defined and Lipschitz continuous from $H^k(\R^N)$ to $\boC^0([0, T], \linebreak H^k (\R^N))$ for any number $0 < T < T_{\max}$.

\noindent $(iii)$ When $\Psi^0 \in H^\ell(\R^N)$, with $\ell > k$, the solution $\Psi$ lies in $\boC^0([0, T], H^\ell(\R^N))$ for any number $0 < T < T_{\max}$.

\noindent $(iv)$ The $L^2$-mass $M_2$ and the cubic Schr\"odinger energy $E_{\rm CS}$ given by
$$M_2(\Psi) = \int_{\R^N} |\Psi|^2, \quad \text{and} \quad E_{\rm CS}(\Psi) = \frac{1}{2} \int_{\R^N} |\nabla \Psi|^2 - \frac{1}{4} \int_{\R^N} |\Psi|^4,$$
are conserved along the flow.
\end{thm*}

The Cauchy problem for the Landau-Lifshitz equation is much more involved. In view of the definition of the Landau-Lifshitz energy, it is natural to solve it in the energy set defined as
$$\boE(\R^N) := \big\{ v \in L_{\rm loc}^1(\R^N, \S^2) : \nabla v \in L^2(\R^N) \ {\rm and} \ (v_1, v_3) \in L^2(\R^N)^2 \big\}.$$
This set appears as a subset of the vector space
$$Z^1(\R^N) := \big\{ v \in L_{\rm loc}^1(\R^N, \R^3) : \nabla v \in L^2(\R^N), v_2 \in L^\infty(\R^N) \ {\rm and} \ (v_1, v_3) \in L^2(\R^N)^2 \big\},$$
which is naturally endowed with the norm
$$\| v \|_{Z^1} := \big( \| v_1 \|_{H^1}^2 + \| v_2 \|_{L^\infty}^2 + \| \nabla v_2 \|_{L^2}^2 + \| v_3 \|_{H^1}^2\big)^\frac{1}{2}.$$
To our knowledge, the well-posedness of the Landau-Lifshitz equation for general initial data in $\boE(\R^N)$ remains an open question. This difficulty is related to the fact that the Landau-Lifshitz equation is an anisotropic perturbation of the Schr\"odinger map equation, and the Cauchy problem for this class of equations is well-known to be intrinsically difficult due to their geometric nature (see e.g.~\cite{BeIoKeT1} and the references therein).

On the other hand, our derivation of the cubic Schr\"odinger equation requires additional smoothness, so that in the sequel, we do not address the Cauchy problem for the Landau-Lifshitz equation in $\boE(\R^N)$. Instead, we focus on the well-posedness for smooth solutions. Given an integer $k \geq 1$, we set 
$$\boE^k(\R^N) := \big\{ v \in \boE(\R^N) : \nabla v \in H^{k - 1}(\R^N) \big\},$$
and we endow this set with the metric structure provided by the norm
$$\| v \|_{Z^k} := \Big( \| v_1 \|_{H^k}^2 + \| v_2 \|_{L^\infty}^2 + \| \nabla v_2 \|_{H^{k - 1}}^2 + \| v_3 \|_{H^k}^2 \big)^\frac{1}{2},$$
of the vector space
\begin{equation}
\label{def:Zk}
Z^k(\R^N) := \big\{ v \in L_{\rm loc}^1(\R^N, \R^3) : (v_1, v_3) \in L^2(\R^N)^2, v_2 \in L^\infty(\R^N) \ {\rm and} \ \nabla v \in H^{k - 1}(\R^N) \big\}.
\end{equation}
Observe that the energy set $\boE(\R^N)$ then identifies with $\boE^1(\R^N)$.

When $k$ is large enough, local well-posedness of the Landau-Lifshitz equation in the set $\boE^k(\R^N)$ follows from the next statement of~\cite{deLaGra2}. 

\begin{thm}[\cite{deLaGra2}]
\label{thm:Cauchy-LL}
Let $\lambda_1$ and $\lambda_3$ be non-negative numbers, and $k \in \N$, with $k > N/2 + 1$. Given any function $m^0 \in \boE^k(\R^N)$, there exist a positive number $T_{\max}$ and a unique solution $m : \R^N \times[0, T_{\max}) \to \S^2$ to the Landau-Lifshitz equation with initial datum $m^0$, which satisfies the following statements.

\noindent $(i)$ The solution $m$ is in the space $L^\infty([0, T], \boE^k(\R^N))$, while its time derivative $\partial_t m$ is in $L^\infty([0, T], H^{k - 2}(\R^N))$, for any number $0 < T < T_{\max}$.

\noindent $(ii)$ If the maximal time of existence $T_{\max}$ is finite, then
\begin{equation}
\label{eq:cond-Tmax-LL}
\int_0^{T_{\max}} \| \nabla m(\cdot, t) \|_{L^\infty}^2 \, dt = \infty.
\end{equation}

\noindent $(iii)$ The flow map $m^0 \mapsto m$ is locally well-defined and Lipschitz continuous from $\boE^k(\R^N)$ to $\boC^0([0, T], \boE^{k - 1}(\R^N))$ for any number $0 < T < T_{\max}$.

\noindent $(iv)$ When $m^0 \in \boE^\ell(\R^N)$, with $\ell > k$, the solution $m$ lies in $L^\infty([0, T], \boE^\ell(\R^N))$, with $\partial_t m \in L^\infty([0, T], H^{\ell - 2}(\R^N))$ for any number $0 < T < T_{\max}$.

\noindent $(v)$ The Landau-Lifshitz energy is conserved along the flow.
\end{thm}

In other words, there exists a unique local continuous flow corresponding to smooth solutions of the Landau-Lifshitz equation. The proof of this property is based on combining a priori energy estimates with a compactness argument. For the Schr\"odinger map equation, the same result was first proved in~\cite{ChaShUh1} when $N = 1$, and in~\cite{McGahag1} for $N \geq 2$ (see also~\cite{ZhouGuo1, SulSuBa1, DingWan1} for the construction of smooth solutions). In the more general context of hyperbolic systems, a similar result is expected when $k > N/2 + 1$ due to the fact that the critical regularity of the equation is given by the condition $k = N/2$ (see e.g.~\cite[Theorem 1.2]{Taylor03}).

Going on with our rigorous derivation of the cubic Schr\"odinger regime, we now express the previous statements in terms of the nonlinear Schr\"odinger equation~\eqref{NLS-eps} satisfied by the rescaled function $\Psi_\varepsilon$.

\begin{cor}
\label{cor:Cauchy-Psi-eps}
Let $\varepsilon$ be a fixed positive number, and $k \in \N$, with $k > N/2 + 1$. Consider a function $\Psi_\varepsilon^0 \in H^k(\R^N)$ such that
\begin{equation}
\label{cond:petit-Psi-eps}
\varepsilon^\frac{1}{2} \, \big\| \Psi_\varepsilon^0 \big\|_{L^\infty} < 1.
\end{equation}
Then, there exist a positive number $T_\varepsilon$ and a unique solution $\Psi_\varepsilon : \R^N \times[0, T_\varepsilon) \to \C$ to~\eqref{NLS-eps} with initial datum $\Psi_\varepsilon^0$, which satisfies the following statements.

\noindent $(i)$ The solution $\Psi_\varepsilon$ is in the space $L^\infty([0, T], H^k(\R^N))$, while its time derivative $\partial_t \Psi_\varepsilon$ is in $L^\infty([0, T], H^{k - 2}(\R^N))$, for any number $0 < T < T_\varepsilon$.

\noindent $(ii)$ If the maximal time of existence $T_\varepsilon$ is finite, then
\begin{equation}
\label{eq:cond-Tmax-NLS-eps}
\int_0^{T_\varepsilon} \big\| \nabla \Psi_\varepsilon(\cdot, t) \big\|_{L^\infty}^2 \, dt = \infty, \quad \text{or} \quad \varepsilon^\frac{1}{2} \, \lim_{t \to T_\varepsilon} \big\| \Psi_\varepsilon(\cdot, t) \big\|_{L^\infty} = 1.
\end{equation}

\noindent $(iii)$ The flow map $\Psi_\varepsilon^0 \mapsto \Psi_\varepsilon$ is locally well-defined and Lipschitz continuous from $H^k(\R^N)$ to $\boC^0([0, T], H^{k - 1}(\R^N))$ for any number $0 < T < T_\varepsilon$.

\noindent $(iv)$ When $\Psi_\varepsilon^0 \in H^\ell(\R^N)$, with $\ell > k$, the solution $\Psi_\varepsilon$ lies in $L^\infty([0, T], H^\ell(\R^N))$, with $\partial_t \Psi_\varepsilon \in L^\infty([0, T], H^{\ell - 2}(\R^N))$ for any number $0 < T < T_\varepsilon$.

\noindent $(v)$ The nonlinear Schr\"odinger energy $\gE_\varepsilon$ given by
$$\gE_\varepsilon(\Psi_\varepsilon) = \frac{1}{2} \int_{\R^N} \bigg( |\Psi_\varepsilon|^2 + \varepsilon |\nabla \Psi_\varepsilon|^2 + \frac{\varepsilon^2 \langle \Psi_\varepsilon, \nabla \Psi_\varepsilon \rangle_\C^2}{1 - \varepsilon |\Psi_\varepsilon|^2} \bigg),$$
is conserved along the flow.

$(vi)$ Set
$$m^0 = \Big( \varepsilon^\frac{1}{2} \, \Re \big( \Psi_\varepsilon^0 \big), \big( 1 - \varepsilon |\Psi_\varepsilon^0|^2 \big)^\frac{1}{2}, \varepsilon^\frac{1}{2} \, \Im \big( \Psi_\varepsilon^0 \big) \Big).$$
The function $m : \R^N \times [0, T_\varepsilon] \to \S^2$ given by
\begin{equation}
\label{eq:m-Psi-eps}
m(x, t) = \Big( \varepsilon^\frac{1}{2} \, \Re \big( e^{- \frac{i t}{\varepsilon}} \Psi_\varepsilon(x, t) \big), \big( 1 - \varepsilon |\Psi_\varepsilon(x, t)|^2 \big)^\frac{1}{2}, \varepsilon^\frac{1}{2} \, \Im \big( e^{- \frac{i t}{\varepsilon}} \Psi_\varepsilon(x, t) \big) \Big),
\end{equation}
for any $(x, t) \in \R^N \times [0, T_\varepsilon]$, is the unique solution to~\eqref{LL} with initial datum $m^0$ of Theorem~\ref{thm:Cauchy-LL}.
\end{cor}

\begin{rem}
\label{rem:inter}
Coming back to the proof of Theorem~\ref{thm:Cauchy-LL} in~\cite{deLaGra2} and using standard interpolation theory, one can check that the flow map $\Psi_\varepsilon^0 \mapsto \Psi_\varepsilon$ is locally well-defined and continuous from $H^k(\R^N)$ to $\boC^0([0, T], H^s(\R^N))$ for any number $0 < T < T_\varepsilon$ and any number $s < k$. 
\end{rem}

Corollary~\ref{cor:Cauchy-Psi-eps} also provides the existence of a unique local continuous flow corresponding to smooth solutions to~\eqref{NLS-eps}. Its proof relies on the equivalence between the Landau-Lifshitz equation and the nonlinear Schr\"odinger equation~\eqref{NLS-eps}, when condition~\eqref{cond:petit-Psi-eps} is satisfied (see statement $(vi)$ above). We refer to Subsection~\ref{sub:Cauchy-Psi-eps} for a detailed proof of this result.

With Corollary~\ref{cor:Cauchy-Psi-eps} at hand, we are now in position to state our main result concerning the rigorous derivation of the cubic Schr\"odinger regime of the Landau-Lifshitz equation.

\begin{thm}
\label{thm:conv-CLS}
Let $0 < \varepsilon < 1$ be a positive number, and $k \in \N$, with $k > N/2 + 2$. Consider two initial conditions $\Psi^0 \in H^k(\R^N)$ and $\Psi_\varepsilon^0 \in H^{k + 3}(\R^N)$, and set
$$\boK_\varepsilon^0 := \big\| \Psi^0 \big\|_{H^k} + \big\| \Psi_\varepsilon^0 \big\|_{H^k} + \varepsilon^\frac{1}{2} \big\| \nabla \Psi_\varepsilon^0 \big\|_{\dot{H}^k} + \varepsilon \big\| \Delta \Psi_\varepsilon^0 \big\|_{\dot{H}^k}.$$
Then, there exists a positive number $A$, depending only on $k$, such that, if the initial data $\Psi^0$ and $\Psi_\varepsilon^0$ satisfy the condition
\begin{equation}
\label{cond:key}
A \, \varepsilon^\frac{1}{2} \, \boK_\varepsilon^0 \leq 1,
\end{equation}
we have the following statements.

\noindent $(i)$ There exists a positive number
\begin{equation}
\label{cond:Teps}
T_\varepsilon \geq \frac{1}{A (\boK_\varepsilon^0)^2},
\end{equation}
such that both the unique solution $\Psi_\varepsilon$ to~\eqref{NLS-eps} with initial datum $\Psi_\varepsilon^0$, and the unique solution $\Psi$ to~\eqref{CS} with initial datum $\Psi^0$ are well-defined on the time interval $[0, T_\varepsilon]$.

\noindent $(ii)$ We have the error estimate
\begin{equation}
\label{eq:est-error}
\big\| \Psi_\varepsilon(\cdot, t) - \Psi(\cdot, t) \big\|_{H^{k - 2}} \leq \Big( \big\| \Psi_\varepsilon^0 - \Psi^0 \big\|_{H^{k - 2}} + A \varepsilon \boK_\varepsilon^0 \big( 1 + (\boK_\varepsilon^0)^3 \big) \Big) \, e^{A (\boK_\varepsilon^0)^2 t},
\end{equation}
for any $0 \leq t \leq T_\varepsilon$.
\end{thm}

Theorem~\ref{thm:conv-CLS} does not only rigorously state the convergence of the Landau-Lifshitz equation towards the cubic Schr\"odinger equation in any dimension. It also quantifies this convergence in the spirit of what we already proved for the Sine-Gordon regime in~\cite{deLaGra2} (see statement $(iv)$ of~\cite[Theorem~1]{deLaGra2}). The assumptions $k>N/2+2$ in Theorem~\ref{thm:conv-CLS} originates in our choice to quantify this convergence. They are taylored in order to obtain the $\varepsilon$ factor in the right-hand side of the error estimate~\eqref{eq:est-error} since we expect this order of convergence to be sharp.

This claim relies on the study of the solitons of the one-dimensional Landau-Lifshitz and cubic Schr\"odinger equations. In Appendix~\ref{sec:solitons}, we classify the solitons $m_{c, \omega}$ with speed $c$ and angular velocity $\omega$ of the Landau-Lifshitz equation when $\lambda = \lambda_1 = \lambda_3$ (see Theorem~\ref{thm:soliton}). We then prove that their difference with respect to the corresponding bright solitons $\Psi_{c, \omega}$ of the cubic Schr\"odinger equation is of exact order $\varepsilon$ as the error factor in~\eqref{eq:est-error} (see Proposition~\ref{prop:diff-solitons}).

It is certainly possible to show only convergence under weaker assumptions by using compactness arguments as for the derivation of similar asymptotic regimes (see e.g.~\cite{ShatZen1, ChirRou2, GermRou1} concerning Schr\"odinger-like equations).

Observe that smooth solutions for both the Landau-Lifshitz and the cubic Schr\"odinger equations are known to exist when the integer $k$ satisfies the condition $k > N/2 + 1$. The additional assumption $k > N/2 + 2$ in Theorem~\ref{thm:conv-CLS} is related to the fact that our proof of~\eqref{eq:est-error} requires a uniform control of the difference $\Psi_\varepsilon - \Psi$, which follows from the Sobolev embedding theorem of $H^{k - 2}(\R^N)$ into $L^\infty(\R^N)$.

Similarly, the fact that $\Psi_\varepsilon^0$ is taken in $H^{k + 3}(\R^N)$ instead of $H^{k + 2}(\R^N)$, which is enough to define the quantity $\boK_\varepsilon^0$, is related to the loss of one derivative for establishing the flow continuity in statement $(iii)$ of Corollary~\ref{cor:Cauchy-Psi-eps}.

Finally, the loss of two derivatives in the error estimate~\eqref{eq:est-error} can be partially recovered by combining standard interpolation theory with the estimates in Proposition~\ref{prop:estim-eps} and Lemma~\ref{lem:estim-CS}. Under the assumptions of Theorem~\ref{thm:conv-CLS}, the solutions $\Psi_\varepsilon$ converge towards the solution $\Psi$ in $\boC^0([0, T_\varepsilon], H^s(\R^N))$ for any $0 \leq s < k$, when $\Psi_\varepsilon^0$ tends to $\Psi^0$ in $H^{k + 2}(\R^N)$ as $\varepsilon \to 0$, but the error term is not necessarily of order $\varepsilon$ due to the interpolation process.

Note here that condition~\eqref{cond:key} is not really restrictive in order to analyze such a convergence. At least when $\Psi_\varepsilon^0$ tends to $\Psi^0$ in $H^{k + 2}(\R^N)$ as $\varepsilon \to 0$, the quantity $\boK_\varepsilon^0$ tends to twice the norm $\| \Psi^0 \|_{H^k}$ in the limit $\varepsilon \to 0$, so that condition~\eqref{cond:key} is always fulfilled. Moreover, the error estimate~\eqref{eq:est-error} is available on a time interval of order $1/\| \Psi^0 \|_{H^k}^2$, which is similar to the minimal time of existence of the smooth solutions to the cubic Schr\"odinger equation (see Lemma~\ref{lem:estim-CS} below).

Apart from the intrinsic interest of Theorem~\ref{thm:conv-CLS}, it is well-known that deriving asymptotic regimes is a powerful tool in order to tackle the analysis of intricate equations. In this direction, we expect that our rigorous derivation of the cubic Schr\"odinger regime will be a useful tool in order to describe the dynamical properties of the Landau-Lifshitz equation, in particular the role played by the solitons in this dynamics (see e.g.~\cite{BetGrSm2, GravSme1} where this strategy was developed in order to prove the asymptotic stability of the dark solitons of the Gross-Pitaevskii equation by using its link with the Korteweg-de Vries equation~\cite{ChirRou2, BeGrSaS2, BeGrSaS3}).

The rest of the paper is mainly devoted to the proof of Theorem~\ref{thm:conv-CLS}. In Section~\ref{sec:strategy}, we explain our strategy for this proof. Section~\ref{sec:details} gathers the proof of Corollary~\ref{cor:Cauchy-Psi-eps}, as well as the detailed proofs of the main steps in the proof of Theorem~\ref{thm:conv-CLS}. Finally, Appendix~\ref{sec:solitons} deals with the classification of the solitons of the Landau-Lifshitz equation when $\lambda = \lambda_1 = \lambda_3$, and with their convergence towards the bright solitons of the cubic Schr\"odinger equation.

%%%%%%%%%%%%%%%%%%%%%%%%%%%%%%%%%%%%%%
%%%%%%%%%%%%%%%%%%%%%%%%%%%%%%%%%%%%%%
%%%%%%%%%%%%%%%%%%%%%%%%%%%%%%%%%%%%%%
\numberwithin{equation}{section}
\numberwithin{thm}{section}
\section{Strategy of the proof of Theorem~\ref{thm:conv-CLS}}
\label{sec:strategy}
%%%%%%%%%%%%%%%%%%%%%%%%%%%%%%%%%%%%%%
%%%%%%%%%%%%%%%%%%%%%%%%%%%%%%%%%%%%%%
%%%%%%%%%%%%%%%%%%%%%%%%%%%%%%%%%%%%%%

The proof relies on the consistency between the Schr\"odinger equations~\eqref{NLS-eps} and~\eqref{CS} in the limit $\varepsilon \to 0$. Indeed, we can recast~\eqref{NLS-eps} as
\begin{equation}
\label{eq:consistency}
i \partial_t \Psi_\varepsilon + \Delta \Psi_\varepsilon + \frac{1}{2} |\Psi_\varepsilon|^2 \Psi_\varepsilon = \varepsilon \boR_\varepsilon,
\end{equation}
where the remainder term $\boR_\varepsilon$ is given by
\begin{equation}
\label{def:R-eps}
\boR_\varepsilon := \frac{|\Psi_\varepsilon|^2}{1 + ( 1 - \varepsilon |\Psi_\varepsilon|^2 )^\frac{1}{2}} \Delta \Psi_\varepsilon - \frac{|\Psi_\varepsilon|^4}{2 (1 + (1 - \varepsilon |\Psi_\varepsilon|^2)^\frac{1}{2})^2} \Psi_\varepsilon - \div \Big( \frac{\langle \Psi_\varepsilon, \nabla \Psi_\varepsilon \rangle_\C}{(1 - \varepsilon |\Psi_\varepsilon|^2)^\frac{1}{2}} \Big) \Psi_\varepsilon .
\end{equation}
In order to establish the convergence towards the cubic Schr\"odinger equation, our main goal is to control the remainder term $\boR_\varepsilon$ on a time interval $[0, T_\varepsilon]$ as long as possible. In particular, we have to show that the maximal time $T_\varepsilon$ for this control does not vanish in the limit $\varepsilon \to 0$.

The strategy for reaching this goal is reminiscent from a series of papers concerning the rigorous derivation of long-wave regimes for various Schr\"odinger-like equations (see~\cite{ShatZen1, BetDaSm1, BeGrSaS2, ChirRou2, BeGrSaS3, BeDaGSS1, Chiron9, GermRou1, deLaGra2} and the references therein). The main argument is to perform suitable energy estimates on the solutions $\Psi_\varepsilon$ to~\eqref{NLS-eps}. These estimates provide Sobolev bounds for the remainder term $\boR_\varepsilon$, which are used to control the differences $u_\varepsilon := \Psi_\varepsilon - \Psi$ with respect to the solutions $\Psi$ to~\eqref{CS}. This further control is also derived from energy estimates. 

Concerning the estimates of the solutions $\Psi_\varepsilon$, we rely on the equivalence with the solutions $m$ to~\eqref{LL} in Corollary~\ref{cor:Cauchy-Psi-eps}. Using this equivalence, we can go back to the computations made in~\cite{deLaGra2} for the derivation of the Sine-Gordon regime of the Landau-Lifshitz equation. More precisely, given a positive number $T$ and a sufficiently smooth solution $m : \R^N \times [0, T] \to \S^2$ to~\eqref{LL}, we define the energy $E_{\rm LL}^k$ of order $k \geq 2$ as
\begin{equation}
\label{def:E-LL-k}
\begin{split}
E_{\rm LL}^k(t) := \frac{1}{2} \Big( & \| \partial_t m(\cdot, t) \|_{\dot{H}^{k - 2}}^2 + \| \Delta m(\cdot, t) \|_{\dot{H}^{k - 2}}^2 + (\lambda_1 + \lambda_3) \big( \| \nabla m_1(\cdot, t) \|_{\dot{H}^{k - 2}}^2\\
& + \| \nabla m_3(\cdot, t) \|_{\dot{H}^{k - 2}}^2 \big) + \lambda_1 \lambda_3 \big( \| m_1(\cdot, t) \|_{\dot{H}^{k - 2}}^2 + \| m_3(\cdot, t) \|_{\dot{H}^{k - 2}}^2 \big) \Big),
\end{split}
\end{equation}
for any $t \in [0, T]$. In the regime $\lambda_1 = \lambda_3 = 1/\varepsilon$, we can prove the following improvement of the computations made in~\cite[Proposition 1]{deLaGra2}.

\begin{prop}
\label{prop:LL-energy-estimate}
Let $0 < \varepsilon < 1$, and $k \in \N$, with $k > N/2 + 1$. Assume that 
\begin{equation}
 \label{cond:lambda}
\lambda_1 = \lambda_3 = \frac{1}{\varepsilon},
\end{equation}
and that $m$ is a solution to~\eqref{LL} in $\boC^0([0, T], \boE^{k + 4}(\R^N))$, with $\partial_t m \in \boC^0([0, T], H^{k + 2}(\R^N))$. Given any integer $2 \leq \ell \leq k + 2$, the energies $E_{\rm LL}^\ell$ are of class $\boC^1$ on $[0, T]$, and there exists a positive number $C_k$, depending possibly on $k$, but not on $\varepsilon$, such that their derivatives satisfy
\begin{equation}
\label{eq:energy-estimate-LL}
\big[ E_{\rm LL}^\ell \big]'(t) \leq \frac{C_k}{\varepsilon} \Big( \| m_1(\cdot, t) \|_{L^\infty}^2 + \| m_3(\cdot, t) \|_{L^\infty}^2 + \| \nabla m(\cdot, t) \|_{L^\infty}^2 \Big) \, \Big( E_{\rm LL}^\ell(t) + E_{\rm LL}^{\ell - 1}(t) \Big),
\end{equation}
for any $t \in [0, T]$. For $\ell - 1 = 1$, the quantity $E_{\rm LL}^1(t)$ in this expression is equal to the Landau-Lifshitz energy $E_{\rm LL}(m(\cdot, t))$.
\end{prop}

As for the proof of~\cite[Proposition 1]{deLaGra2}, the estimates in Proposition~\ref{prop:LL-energy-estimate} rely on the following identity. Under assumption~\eqref{cond:lambda}, we derive from~\eqref{LL} the second-order equation
\begin{equation}
\label{eq:second-LL}
\partial_{tt} m + \Delta^2 m - \frac{2}{\varepsilon} \Big( \Delta m_1 e_1 + \Delta m_3 e_3 \Big) 
+ \frac{1}{\varepsilon^2} \Big( m_1 e_1 + m_3 e_3 \Big) = F_\varepsilon(m),
\end{equation}
where 
\begin{equation}
\label{usa}
\begin{split}
F_\varepsilon(m) := \sum_{1 \leq i, j \leq N} & \Big( \partial_i \big( 2 \langle \partial_i m, \partial_j m \rangle_{\R^3} \partial_j m - |\partial_j m|^2 \partial_i m \big) - 2 \partial_{ij} \big( \langle \partial_i m, \partial_j m \rangle_{\R^3} m \big) \Big)\\
- \frac{1}{\varepsilon} \Big( & (m_1^2 + 3 m_3^2) \Delta m_1 e_1 + (3 m_1^2 + m_3^2) \Delta m_3 e_3 - 2 m_1 m_3 (\Delta m_1 e_3 + \Delta m_3 e_1)\\
& + (m_1^2 + m_3^2) \Delta m_2 e_2 - |\nabla m|^2 (m_1 e_1 + m_3 e_3) + \nabla \big( m_1^2 + m_3^2 \big) \cdot \nabla m \Big)\\
+ \frac{1}{\varepsilon^2} & \Big( (m_1^2+ m_3^2) (m_1 e_1+m_3 e_3) \Big).
\end{split}
\end{equation}
Note here that the computation of this formula uses the pointwise identities
$$\langle m, \partial_i m \rangle_{\R^3} = \langle m, \partial_{ii} m \rangle_{\R^3} + |\partial_i m|^2 = \langle m, \partial_{iij} m \rangle_{\R^3} + 2 \langle \partial_i m, \partial_{ij} m \rangle_{\R^3} + \langle \partial_j m, \partial_{ii} m \rangle_{\R^3} = 0,$$
which hold for any $1 \leq i, j \leq N$, due to the property that $m$ is valued into the sphere $\S^2$.

Since $\lambda_1 = \lambda_3$, the expression of the function $F_\varepsilon(m)$ in~\eqref{usa} is simpler than the one that was computed in~\cite{deLaGra2}. In contrast with the formula in~\cite[Proposition 1]{deLaGra2}, the multiplicative factor in the right-hand side of~\eqref{eq:energy-estimate-LL} now only depends on the uniform norms of the functions $m_1$, $m_3$ and $\nabla m$. This property is crucial in order to use these estimates in the cubic Schr\"odinger regime.

The next step of the proof is indeed to express the quantities $E_{\rm LL}^k$ in terms of the functions $\Psi_\varepsilon$. Assume that these functions $\Psi_\varepsilon : \R^N \times [0, T] \to \C$ are smooth enough. In view of~\eqref{def:Psi-eps} and~\eqref{NLS-eps}, it is natural to set
\begin{equation}
 \label{def:gE-NLS}
\begin{split}
\gE_\varepsilon^k&(t) := \frac{1}{2} \bigg( \big\| \Psi_\varepsilon(\cdot, t) \big\|_{\dot{H}^{k - 2}}^2 + \big\| \varepsilon \partial_t \Psi_\varepsilon(\cdot, t) - i \Psi_\varepsilon(\cdot, t) \big\|_{\dot{H}^{k - 2}}^2 + \varepsilon^2 \big\| \Delta \Psi_\varepsilon(\cdot, t) \big\|_{\dot{H}^{k - 2}}^2\\
+ & \varepsilon \Big( \big\| \partial_t (1 - \varepsilon |\Psi_\varepsilon(\cdot, t)|^2)^\frac{1}{2} \big\|_{\dot{H}^{k - 2}}^2 + \big\| \Delta (1 - \varepsilon |\Psi_\varepsilon(\cdot, t)|^2)^\frac{1}{2} \big\|_{\dot{H}^{k - 2}}^2 + 2 \big\| \nabla \Psi_\varepsilon(\cdot, t) \big\|_{\dot{H}^{k - 2}}^2 \Big) \bigg),
\end{split}
\end{equation}
for any $k \geq 2$ and any $t \in [0, T]$. Combining the local well-posedness result of Corollary~\ref{cor:Cauchy-Psi-eps} and the computations in Proposition~\ref{prop:LL-energy-estimate}, we obtain

\begin{cor}
\label{cor:NLS-energy-estimate}
Let $0 < \varepsilon < 1$, and $k \in \N$, with $k > N/2 + 1$. Consider a function $\Psi_\varepsilon^0 \in H^{k + 5}(\R^N)$ satisfying condition~\eqref{cond:petit-Psi-eps}, and let $\Psi_\varepsilon : \R^N \times[0, T_\varepsilon) \to \C$ be the corresponding solution to~\eqref{NLS-eps} given by Corollary~\ref{cor:Cauchy-Psi-eps}. Given any integer $2 \leq \ell \leq k + 2$ and any number $0 \leq T < T_\varepsilon$, the energies $\gE_\varepsilon^\ell$ are of class $\boC^1$ on $[0, T]$, and there exists a positive number $C_k$, depending possibly on $k$, but not on $\varepsilon$, such that their derivatives satisfy
\begin{equation}
\label{eq:energy-estimate-NLS}
\big[ \gE_\varepsilon^\ell \big]'(t) \leq C_k \bigg( \| \Psi_\varepsilon(\cdot, t) \|_{L^\infty}^2 + \| \nabla \Psi_\varepsilon(\cdot, t) \|_{L^\infty}^2 + \varepsilon \Big\| \frac{\langle \Psi_\varepsilon(\cdot, t), \nabla \Psi_\varepsilon(\cdot, t) \rangle_\C}{(1 - \varepsilon |\Psi_\varepsilon(\cdot, t)|^2)^\frac{1}{2}} \Big\|_{L^\infty}^2 \bigg) \, \big( \gE_\varepsilon^\ell(t) + \varepsilon^{\delta_{\ell, 2}} \gE_\varepsilon^{\ell - 1}(t) \big),
\end{equation}
for any $t \in [0, T]$. For $\ell - 1 = 1$, the quantity $\gE_\varepsilon^1(t)$ in this expression is equal to the nonlinear Schr\"odinger energy $\gE_\varepsilon(\Psi_\varepsilon(\cdot, t))$.
\end{cor}

In order to gain a control on the solutions $\Psi_\varepsilon$ to~\eqref{NLS-eps} from inequality~\eqref{eq:energy-estimate-NLS}, we now have to characterize the Sobolev norms, which are controlled by the energies $\gE_\varepsilon^k$. In this direction, we show

\begin{lem}
\label{lem:cont-energies-NLS}
Let $0 < \varepsilon < 1$, $T > 0$ and $k \in \N$, with $k > N/2 + 1$. Consider a solution $\Psi_\varepsilon \in \boC^0([0, T], H^{k + 4}(\R^N))$ to~\eqref{NLS-eps} such that
\begin{equation}
\label{eq:cond-Psi-eps-T}
\sigma_T := \varepsilon^\frac{1}{2} \max_{t \in [0, T]} \big\| \Psi_\varepsilon(\cdot, t) \big\|_{L^\infty} < 1.
\end{equation}
There exists a positive number $C$, depending possibly on $\sigma_T$ and $k$, but not on $\varepsilon$, such that
\begin{equation}
 \label{eq:cont-gEl-NLS}
\begin{split}
\frac{1}{2} \Big( \big\| \Psi_\varepsilon(\cdot, t) \big\|_{\dot{H}^{\ell - 2}}^2 & + \varepsilon \big\| \nabla \Psi_\varepsilon(\cdot, t) \big\|_{\dot{H}^{\ell - 2}}^2 + \varepsilon^2 \big\| \Delta \Psi_\varepsilon(\cdot, t) \big\|_{\dot{H}^{\ell - 2}}^2 \Big)\\
& \leq \gE_\varepsilon^\ell(t) \leq C \Big( \big\| \Psi_\varepsilon(\cdot, t) \big\|_{\dot{H}^{\ell - 2}}^2 + \varepsilon \big\| \nabla \Psi_\varepsilon(\cdot, t) \big\|_{\dot{H}^{\ell - 2}}^2 + \varepsilon^2 \big\| \Delta \Psi_\varepsilon(\cdot, t) \big\|_{\dot{H}^{\ell - 2}}^2 \Big),
\end{split}
\end{equation}
for any $2 \leq \ell \leq k + 2$ and any $t \in [0, T]$. Moreover, we also have
\begin{equation}
 \label{eq:cont-gE1-NLS}
\begin{split}
& \frac{1}{2} \Big( \big\| \Psi_\varepsilon(\cdot, t) \big\|_{L^2}^2 + \varepsilon \big\| \nabla \Psi_\varepsilon(\cdot, t) \big\|_{L^2}^2 \Big) \leq \gE_\varepsilon^1(t) \leq C \Big( \big\| \Psi_\varepsilon(\cdot, t) \big\|_{L^2}^2 + \varepsilon \big\| \nabla \Psi_\varepsilon(\cdot, t) \big\|_{L^2}^2 \Big),
\end{split}
\end{equation}
for any $t \in [0, T]$.
\end{lem}

With Corollary~\ref{cor:NLS-energy-estimate} and Lemma~\ref{lem:cont-energies-NLS} at hand, we are now in position to provide the following control on the solutions $\Psi_\varepsilon$ to~\eqref{NLS-eps}.

\begin{prop}
\label{prop:estim-eps}
Let $0 < \varepsilon < 1$, $0 < \sigma < 1$ and $k \in \N$, with $k > N/2 + 1$. There exists a positive number $C_*$, depending possibly on $\sigma$ and $k$, but not on $\varepsilon$, such that if an initial datum $\Psi_\varepsilon^0 \in H^{k + 3}(\R^N)$ satisfies 
\begin{equation}
\label{cond:Psi-eps-petit-0}
C_* \varepsilon^\frac{1}{2} \Big( \big\| \Psi_\varepsilon^0 \big\|_{H^k} + \varepsilon^\frac{1}{2} \big\| \nabla \Psi_\varepsilon^0 \big\|_{\dot{H}^k} + \varepsilon \big\| \Delta \Psi_\varepsilon^0 \big\|_{\dot{H}^k} \Big) \leq 1, 
\end{equation}
then there exists a positive time
$$T_\varepsilon \geq \frac{1}{C_* \big( \big\| \Psi_\varepsilon^0 \big\|_{H^k} + \varepsilon^\frac{1}{2} \big\| \nabla \Psi_\varepsilon^0 \big\|_{\dot{H}^k} + \varepsilon \big\| \Delta \Psi_\varepsilon^0 \big\|_{\dot{H}^k} \big)^2},$$
such that the unique solution $\Psi_\varepsilon$ to~\eqref{NLS-eps} with initial condition $\Psi_\varepsilon^0$ satisfies the uniform bound
$$\varepsilon^\frac{1}{2} \big\| \Psi_\varepsilon(\cdot, t) \big\|_{L^\infty} \leq \sigma,$$
as well as the energy estimate
\begin{equation}
\label{eq:borne-Psi-eps}
\begin{split}
\big\| \Psi_\varepsilon(\cdot, t) \big\|_{H^k} + \varepsilon^\frac{1}{2} \big\| \nabla \Psi_\varepsilon(\cdot, t) \big\|_{\dot{H}^k} & + \varepsilon \big\| \Delta \Psi_\varepsilon(\cdot, t) \big\|_{\dot{H}^k}\\
& \leq C_* \Big( \big\| \Psi_\varepsilon^0 \big\|_{H^k} + \varepsilon^\frac{1}{2} \big\| \nabla \Psi_\varepsilon^0 \big\|_{\dot{H}^k} + \varepsilon \big\| \Delta \Psi_\varepsilon^0 \big\|_{\dot{H}^k} \Big),
\end{split}
\end{equation}
for any $0 \leq t \leq T_\varepsilon$.
\end{prop}

An important feature of Proposition~\ref{prop:estim-eps} lies in the fact that the solutions $\Psi_\varepsilon$ are controlled uniformly with respect to the small parameter $\varepsilon$ up to a loss of three derivatives. This loss is usual in the context of asymptotic regimes for Schr\"odinger-like equations (see e.g~\cite{BeGrSaS2, BeGrSaS3} and the references therein). It is related to the property that the energies $E_{\rm LL}^k$ naturally scale according to the right-hand side of~\eqref{eq:borne-Psi-eps} in the limit $\varepsilon \to 0$. This property is the origin of a loss of two derivatives. The extra loss is due to the requirement to use the continuity of the~\eqref{NLS-eps} flow with respect to the initial datum in order to prove Proposition~\ref{prop:estim-eps}, and this continuity holds with a loss of one derivative~\footnote{In view of Remark~\ref{rem:inter}, continuity actually holds with any positive loss of fractional derivatives, which translates by a loss of at least one classical derivative.} in view of statement $(iii)$ in Corollary~\ref{cor:Cauchy-Psi-eps}.

We now turn to our ultimate goal, which is to estimate the error between a solution $\Psi_\varepsilon$ to~\eqref{NLS-eps} and a solution $\Psi$ to~\eqref{CS}. Going back to~\eqref{eq:consistency}, we check that their difference $u_\varepsilon := \Psi_\varepsilon - \Psi$ satisfies the equation 
\begin{equation}
\label{eq:diff}
 i \partial_t u_\varepsilon + \Delta u_\varepsilon + \frac{1}{2} \big( |u_\varepsilon + \Psi|^2 (u_\varepsilon + \Psi) - |\Psi|^2 \Psi \big) = \varepsilon \boR_\varepsilon.
\end{equation}
In view of~\eqref{def:R-eps}, we can invoke Proposition~\ref{prop:estim-eps} in order to bound the remainder term $\boR_\varepsilon$ in suitable Sobolev norms. On the other hand, we also have to provide a Sobolev control of the solution $\Psi$ to~\eqref{CS} on a time interval as long as possible. In this direction, we can show the following classical result (see e.g.~\cite{Cazenav0}), by performing standard energy estimates on the $H^k$-norms of the solution $\Psi$.

\begin{lem}
\label{lem:estim-CS}
Let $k \in \N$, with $k > N/2$, and $\Psi^0 \in H^k(\R^N)$. There exists a positive number $C_k$, depending possibly on $k$, such that there exists a positive time
$$T_* \geq \frac{1}{C_k \| \Psi^0 \|_{H^k}^2},$$
for which the unique solution $\Psi$ to~\eqref{CS} with initial condition $\Psi^0$ satisfies the energy estimate
$$\big\| \Psi(\cdot, t) \big\|_{H^k} \leq C_k \big\| \Psi^0 \big\|_{H^k},$$
for any $0 \leq t \leq T_*$.
\end{lem}

Finally, we can perform standard energy estimates in order to control the difference $u_\varepsilon$ according to the following statement.

\begin{prop}
\label{prop:error}
Let $0 < \varepsilon < 1$, $0 < \sigma < 1$ and $k \in \N$, with $k > N/2 + 2$. Given an initial condition $\Psi_\varepsilon^0 \in H^{k + 3}(\R^N)$, assume that the unique corresponding solution $ \Psi_\varepsilon$ to~\eqref{NLS-eps} is well-defined on a time interval $[0, T]$ for some positive number $T$, and that it satisfies the uniform bound
\begin{equation}
\label{eq:unif-bound}
\varepsilon^\frac{1}{2} \, \big\| \Psi_\varepsilon(\cdot, t) \big\|_{L^\infty} \leq \sigma,
\end{equation}
for any $t \in [0, T]$. Assume similarly that the solution $\Psi$ to~\eqref{CS} with initial datum $\Psi^0\in H^k(\R^N)$ is well-defined on $[0, T]$. Set $u_\varepsilon := \Psi_\varepsilon - \Psi$ and
$$\boK_\varepsilon(T) := \big\| \Psi \big\|_{\boC^0([0, T], H^k)} + \big\| \Psi_\varepsilon \big\|_{\boC^0([0, T], H^k)} + \varepsilon^\frac{1}{2} \big\| \nabla \Psi_\varepsilon \big\|_{\boC^0([0, T], \dot{H}^k)} + \varepsilon \big\| \Delta \Psi_\varepsilon \big\|_{\boC^0([0, T], \dot{H}^k)}.$$
Then there exists a positive number $C_*$, depending possibly on $\sigma$ and $k$, but not on $\varepsilon$, such that
\begin{equation}
 \label{est:error}
 \big\| u_\varepsilon(\cdot, t) \big\|_{H^{k - 2}} \leq \Big( \big\| u_\varepsilon(\cdot, 0) \big\|_{H^{k - 2}} + \varepsilon \boK_\varepsilon(T) \big( 1 + \boK_\varepsilon(T)^3 \big) \Big) \, e^{C_* \boK_\varepsilon(T)^2 t}, 
\end{equation}
for any $t\in [0, T]$.
\end{prop}

We are now in position to conclude the proof of Theorem~\ref{thm:conv-CLS}.

\begin{proof}[Proof of Theorem~\ref{thm:conv-CLS}]
Under the assumptions of Theorem~\ref{thm:conv-CLS}, we can apply Proposition~\ref{prop:estim-eps} (with $\sigma = 1/2$ for instance) and Lemma~\ref{lem:estim-CS}. They provide the existence of a positive number $C_1$, depending only on $k$, and a corresponding number $T_\varepsilon$, with
$$T_\varepsilon \geq \frac{1}{C_1 (\boK_\varepsilon^0)^2},$$
such that, under assumption~\eqref{cond:key} (with $A$ replaced by $C_1$), the solution $\Psi_\varepsilon$ to~\eqref{NLS-eps} with initial datum $\Psi_\varepsilon^0$, and the solution $\Psi$ to~\eqref{CS} with initial datum $\Psi^0$ are well-defined on the time interval $[0, T_\varepsilon]$. Moreover, the function $\Psi_\varepsilon$ satisfies condition~\eqref{eq:unif-bound} with $\sigma = 1/2$ on $[0, T_\varepsilon]$, and the quantity $\boK_\varepsilon(T_\varepsilon)$ in Proposition~\ref{prop:error} is controlled by
\begin{equation}
 \label{eq:cont-K-eps}
\boK_\varepsilon(T_\varepsilon) \leq C_1 \boK_\varepsilon^0.
\end{equation}
As a consequence, we can invoke Proposition~\ref{prop:error} with $\sigma = 1/2$, which gives the existence of another positive number $C_2$, depending only on $k$, such that
$$ \big\| \Psi_\varepsilon(\cdot, t) - \Psi(\cdot, t) \big\|_{H^{k - 2}} \leq \Big( \big\| \Psi_\varepsilon^0 - \Psi^0 \big\|_{H^{k - 2}} + \varepsilon \boK_\varepsilon(T_\varepsilon) \big( 1 + \boK_\varepsilon(T_\varepsilon)^3 \big) \Big) \, e^{C_2 \boK_\varepsilon(T_\varepsilon)^2 t},$$
for any $t\in [0, T_\varepsilon]$. Statement $(ii)$ in Theorem~\ref{thm:conv-CLS} follows, with $A = \max \{ C_1, C_1^4, C_1^2 C_2 \}$. This completes the proof of Theorem~\ref{thm:conv-CLS}.
\end{proof}

%%%%%%%%%%%%%%%%%%
%%%%%%%%%%%%%%%%%%
%%%%%%%%%%%%%%%%%%
\section{Details of the proofs}
\label{sec:details}
%%%%%%%%%%%%%%%%%%
%%%%%%%%%%%%%%%%%%
%%%%%%%%%%%%%%%%%%

%%%%%%%%%%%%%%%%%%%%%%%%%%%%%%%%%%
%%%%%%%%%%%%%%%%%%%%%%%%%%%%%%%%%%
\subsection{Proof of Corollary~\ref{cor:Cauchy-Psi-eps}}
\label{sub:Cauchy-Psi-eps}
%%%%%%%%%%%%%%%%%%%%%%%%%%%%%%%%%%
%%%%%%%%%%%%%%%%%%%%%%%%%%%%%%%%%%

The proof essentially reduces to translate the statements in Theorem~\ref{thm:Cauchy-LL} in terms of the nonlinear Schr\"odinger equation~\eqref{NLS-eps}. Indeed, consider an initial datum $\Psi_\varepsilon^0 \in H^k(\R^N)$ satisfying the assumptions of Corollary~\ref{cor:Cauchy-Psi-eps}. Coming back to the scaling in~\eqref{def:Psi-eps}, we set
$$m_1^0 = \varepsilon^\frac{1}{2} \, \Re \big( \Psi_\varepsilon^0 \big) \quad \text{and} \quad m_3^0 = \varepsilon^\frac{1}{2} \, \Im \big( \Psi_\varepsilon^0 \big),$$
while we can define the function
\begin{equation}
\label{def:m20}
m_2^0 = \big( 1 - \varepsilon |\Psi_\varepsilon^0|^2 \big)^\frac{1}{2},
\end{equation}
due to condition~\eqref{cond:petit-Psi-eps}. The initial datum $m^0 = (m_1^0, m_2^0, m_3^0)$ then lies in $\boE^k(\R^N)$. Hence, there exists a positive number $T_{\max}$ and a unique solution $m$ to~\eqref{LL} with initial datum $m^0$, which satisfies the five statements in Theorem~\ref{thm:Cauchy-LL}. As in~\eqref{def:Psi-eps}, we then set
$$\Psi_\varepsilon(x, t) = \varepsilon^{-\frac{1}{2}}e^\frac{i t}{\varepsilon} \, \big( m_1(x, t) + i \, m_3(x, t) \big).$$
In view of statement $(iii)$ in Theorem~\ref{thm:Cauchy-LL}, the Sobolev embedding theorem guarantees that the function $\Psi_\varepsilon$ belongs to $\boC^0([0, T_{\max}), \boC_b^0(\R^N))$. In particular, we are able to define the number
$$T_\varepsilon := \sup \big\{ T \in [0, T_{\max}) \text{ such that } \varepsilon^\frac{1}{2} \, \| \Psi_\varepsilon(\cdot, t) \|_{L^\infty} < 1 \text{ for any } t < T \big\} \leq T_{\max},$$
and this number is positive due to condition~\eqref{cond:petit-Psi-eps}. Note here that either $T_\varepsilon = T_{\max}$, or
$$\varepsilon^\frac{1}{2} \, \lim_{t \to T_\varepsilon} \big\| \Psi_\varepsilon(\cdot, t) \big\|_{L^\infty} = 1.$$

With these definitions at hand, statements $(i)$, $(ii)$ and $(iv)$ of Corollary~\ref{cor:Cauchy-Psi-eps} literally follow from the same statements in Theorem~\ref{thm:Cauchy-LL}. Moreover, the function $m_2$ is continuous on $\R^N \times [0, T_\varepsilon)$ by statement $(iii)$ in Theorem~\ref{thm:Cauchy-LL}, and it does not vanish on this set by definition of the number $T_\varepsilon$ due to the fact that $m$ is valued in $\S^2$. Therefore, we deduce from~\eqref{def:m20} that
\begin{equation}
\label{eq:m2}
m_2 = \big( 1 - \varepsilon |\Psi_\varepsilon|^2 \big)^\frac{1}{2},
\end{equation}
on $\R^N \times [0, T_\varepsilon)$. Recall here that the map $m$ is solution to~\eqref{LL}. Since this equation holds with $m \in L^\infty([0, T_{\max}), \boE^k(\R^N))$ and $\partial_t m \in L^\infty([0, T_{\max}), H^{k - 2}(\R^N))$, we can check that the functions $\check{m}$ and $m_2$ solve~\eqref{eq:sys-LL}, and it is enough to apply the chain rule theorem to $m_2$ in order to derive that the function $\Psi_\varepsilon$ is solution to~\eqref{NLS-eps} with initial datum $\Psi_\varepsilon^0$. A direct computation then provides the identity
\begin{equation}
\label{eq:scaled-energy}
|\nabla m|^2 + \frac{1}{\varepsilon} \Big( m_1^2 + m_3^2 \Big) = |\Psi_\varepsilon|^2 + \varepsilon |\nabla \Psi_\varepsilon|^2 + \frac{\varepsilon^2 \langle \Psi_\varepsilon, \nabla \Psi_\varepsilon \rangle_\C^2}{1 - \varepsilon |\Psi_\varepsilon|^2},
\end{equation}
and the conservation of the energy $\gE_\varepsilon$ in statement $(v)$ follows from the same statement in Theorem~\ref{thm:Cauchy-LL}. Note finally that this construction of solution $\Psi_\varepsilon$ to~\eqref{NLS-eps} guarantees the local Lipschitz continuity of the flow map in statement $(iv)$ of Corollary~\ref{cor:Cauchy-Psi-eps} due to the same property in Theorem~\ref{thm:Cauchy-LL}.

Concerning uniqueness, the argument is similar. Given another possible solution $\tilde{\Psi}_\varepsilon$ to~\eqref{NLS-eps} with initial data $\Psi_\varepsilon^0$ as in Corollary~\ref{cor:Cauchy-Psi-eps}, we set
$$\tilde{m}(x, t) = \Big( \varepsilon^\frac{1}{2} \, \Re \big( e^{-\frac{i t}{\varepsilon}} \tilde{\Psi}_\varepsilon(x, t) \big), \big( 1 - \varepsilon |\tilde{\Psi}_\varepsilon(x, t)|^2 \big)^\frac{1}{2}, \varepsilon^\frac{1}{2} \, \Im \big( e^{-\frac{i t}{\varepsilon}} \tilde{\Psi}_\varepsilon(x, t) \big) \Big),$$
and we check that the maps $\tilde{m}$ are solutions to~\eqref{LL} and that they satisfy the statements in Theorem~\ref{thm:Cauchy-LL} (on time intervals of the form $[0, \tilde{T}_\varepsilon)$). In view of the uniqueness statement in Theorem~\ref{thm:Cauchy-LL}, and of the previous construction of the solution $\Psi_\varepsilon$ to~\eqref{NLS-eps}, we obtain that $\tilde{\Psi}_\varepsilon = \Psi_\varepsilon$, that is the uniqueness of the solution $\Psi_\varepsilon$. Statement $(vi)$ follows from the same argument. This concludes the proof of Corollary~\ref{cor:Cauchy-Psi-eps}. \qed

%%%%%%%%%%%%%%%%%%%%%%%%%%%%%%%%%%%%%%%
%%%%%%%%%%%%%%%%%%%%%%%%%%%%%%%%%%%%%%%
\subsection{Proof of Proposition~\ref{prop:LL-energy-estimate}}
\label{sub:LL-est}
%%%%%%%%%%%%%%%%%%%%%%%%%%%%%%%%%%%%%%%
%%%%%%%%%%%%%%%%%%%%%%%%%%%%%%%%%%%%%%%

Let us first recall the Moser estimates
\begin{equation}
\label{eq:Moser}
\big\| \partial_x^{\alpha_1} f_1 \partial_x^{\alpha_2} f_2 \cdots \partial_x^{\alpha_j} f_j \big\|_{L^2} \leq C_{j, k} \max_{1 \leq i \leq j} \prod_{m \neq i} \big\| f_m \big\|_{L^\infty} \, \big\| f_i \big\|_{\dot{H}^\ell},
\end{equation}
which hold for any integers $(j, \ell) \in \N^2$, any $\alpha =(\alpha_1, \ldots, \alpha_j) \in \N^j$, with $\sum_{i = 1}^j \alpha_i = \ell$, and any functions $(f_1, \dots ,f_j) \in L^\infty(\R^N)^j \cap \dot{H}^\ell(\R^N)^j$ (see e.g.~\cite{Moser1, Hormand0}). Under the assumptions of Proposition~\ref{prop:LL-energy-estimate}, we derive from these estimates that the second-order derivative $\partial_{tt} m$ belongs to $\boC^0([0, T], H^k(\R^N))$. Hence, the energies $E_{\rm LL}^\ell$ are of class $\boC^1$ on $[0, T]$. Moreover, in view of~\eqref{eq:second-LL}, integrating by parts provides the formula
\begin{equation}
\label{arkansas}
\big[ E_{\rm LL}^\ell \big]'(t) = \big\langle \partial_t m(\cdot, t), F_\varepsilon(m)(\cdot, t) \big\rangle_{\dot{H}^{\ell - 2}} = \sum_{|\alpha| = \ell - 2} \int_{\R^N} \big\langle \partial_t \partial_x^\alpha m(x, t), \partial_x^\alpha F_\varepsilon(m)(x, t) \big\rangle_{\R^3} \, dx,
\end{equation}
for any $t \in [0, T]$. In order to obtain~\eqref{eq:energy-estimate-LL}, we now have to control the derivatives $\partial_x^\alpha F_\varepsilon(m)$ with respect to the various terms in the quantities $E_{\rm LL}^\ell$ and $E_{\rm LL}^{\ell - 1}$. Here, we face the difficulty that the derivative $\partial_x^\alpha F_\varepsilon(m)$ contains partial derivatives of order $\ell + 1$ of the function $m$, which cannot be bounded with respect to the quantities $E_{\rm LL}^\ell$ and $E_{\rm LL}^{\ell - 1}$. In order to by-pass this difficulty, we bring to light a hidden geometric cancellation in the scalar product in the right-hand side of~\eqref{arkansas}.

Let us decompose the derivative $\partial_x^\alpha F_\varepsilon(m)$ as
\begin{equation}
\label{arizona}
\partial_x^\alpha F_\varepsilon(m) = - 2 \sum_{1 \leq i, j \leq N} \partial_{i j} \big( \langle \partial_i m, \partial_j m \rangle_{\R^3} \big) m + G_1(m) - \frac{1}{\varepsilon} G_2(m) + \frac{1}{\varepsilon^2} G_3(m),
\end{equation}
according to the definition of $F_\varepsilon(m)$ in~\eqref{usa}. By the Leibniz formula, the quantity $G_1(m)$ is here given by
\begin{align*}
G_1(m) := \sum_{1 \leq i, j \leq N} \Big( & \partial_x^\alpha \partial_i \big( 2 \langle \partial_i m, \partial_j m \rangle_{\R^3} \partial_j m - |\partial_j m|^2 \partial_i m \big)\\
& - 2 \sum _{0 \leq \beta \leq \tilde{\alpha}, \beta \neq 0} \binom{\tilde{\alpha}}{\beta} \partial_x^{\tilde{\alpha} - \beta} \big( \langle \partial_i m, \partial_j m \rangle_{\R^3} \big) \partial_x^\beta m \Big),
\end{align*}
where $\partial_x^{\tilde{\alpha}} := \partial_x^\alpha \partial_{i j}$. As a consequence, we directly infer from the Leibniz formula and the Moser estimates~\eqref{eq:Moser} that
\begin{equation}
\label{missouri}
\big\| G_1(m) \big\|_{L^2} \leq C_k \, \| \nabla m \|_{L^\infty}^2 \, \| \nabla m \|_{\dot{H}^{\ell - 1}}.
\end{equation}
Here as in the sequel, the notation $C_k$ refers to a positive number depending only on $k$. Observe that the uniform boundedness of the gradient $\nabla m$ is a consequence of the Sobolev embedding theorem and the assumption $k > N/2 + 1$. Similarly, we have
\begin{equation}
\label{wyoming}
\begin{split}
\big\| G_2(m) \big\|_{L^2} \leq C_k \bigg( & \| \nabla m \|_{\dot{H}^{\ell - 1}} \, \big( \| m_1 \|_{L^\infty}^2 + \| m_3 \|_{L^\infty}^2 \big) + \| \nabla m \|_{L^\infty}^2 \, \big( \| m_1 \|_{\dot{H}^{\ell - 2}} + \| m_3 \|_{\dot{H}^{\ell - 2}} \big)\\
& + \| \nabla m \|_{L^\infty} \, \| \nabla m \|_{\dot{H}^{\ell - 2}} \, \big( \| m_1 \|_{L^\infty} + \| m_3 \|_{L^\infty} \big)\bigg),
\end{split}
\end{equation}
for the quantity
\begin{align*}
G_2(m) := & \partial_x^\alpha \Big( (m_1^2 + 3 m_3^2) \Delta m_1 e_1 + (3 m_1^2 + m_3^2) \Delta m_3 e_3 - 2 m_1 m_3 (\Delta m_1 e_3 + \Delta m_3 e_1)\\
& + (m_1^2 + m_3^2) \Delta m_2 e_2 - |\nabla m|^2 (m_1 e_1 + m_3 e_3) + \nabla \big( m_1^2 + m_3^2 \big) \cdot \nabla m \Big),
\end{align*}
and
\begin{equation}
\label{nevada}
\big\| G_3(m) \big\|_{L^2} \leq C_k \big( \| m_1 \|_{L^\infty}^2 + \| m_3 \|_{L^\infty}^2 \big) \, \big( \| m_1 \|_{\dot{H}^{\ell - 2}} + \| m_3 \|_{\dot{H}^{\ell - 2}} \big),
\end{equation}
where
$$G_3(m) := \partial_x^\alpha \big( (m_1^2+ m_3^2) (m_1 e_1+m_3 e_3) \big).$$

We then deal with the remaining term of decomposition~\eqref{arizona}. Coming back to~\eqref{arkansas}, we integrate by parts in order to write
\begin{equation}
\label{carolina}
\begin{split}
\int_{\R^N} \Big\langle \partial_t \partial_x^\alpha m, & \partial_x^\alpha \partial_{ij} \big( \langle \partial_i m, \partial_j m \rangle_{\R^3} \big) m \Big\rangle_{\R^3}\\
& = - \int_{\R^N} \partial_x^\alpha \partial_j \big( \langle \partial_i m, \partial_j m \rangle_{\R^3} \big) \Big( \big\langle \partial_t \partial_x^\alpha \partial_i m, m \big\rangle_{\R^3} + \big\langle \partial_t \partial_x^\alpha m, \partial_i m \big\rangle_{\R^3} \Big),
\end{split}
\end{equation}
for any $1 \leq i, j \leq N$. Arguing as before, we first obtain
\begin{equation}
\label{connecticut}
\bigg| \int_{\R^N} \partial_x^\alpha \partial_j \big( \langle \partial_i m, \partial_j m \rangle_{\R^3} \big) \, \big\langle \partial_t \partial_x^\alpha m, \partial_i m \big\rangle_{\R^3} \bigg| \leq C_k \, \| \nabla m \|_{L^\infty}^2 \, \| \nabla m \|_{\dot{H}^{\ell - 1}} \, \| \partial_t \partial_x^\alpha m \|_{L^2}.
\end{equation}
On the other hand, it follows from the Landau-Lifshitz equation and the Leibniz formula that
\begin{align*}
\partial_x^\alpha & \partial_j \big( \langle \partial_i m, \partial_j m \rangle_{\R^3} \big) \, \big\langle \partial_t \partial_x^{\alpha^*} m, m \big\rangle_{\R^3}\\
= & - \sum_{\beta \leq \alpha^*} \binom{\alpha^*}{\beta} \partial_x^\alpha \partial_j \big( \langle \partial_i m, \partial_j m \rangle_{\R^3} \big) \, \big\langle \partial_x^\beta m \times \partial_x^{\alpha^* - \beta} \big( \Delta m - \frac{1}{\varepsilon} \big( m_1 e_1 +m_3 e_3 \big) \big), m \big\rangle_{\R^3},
\end{align*}
where $\partial_x^{\alpha^*} := \partial_x^\alpha \partial_i$. The right-hand side of this formula vanishes when $\beta = 0$. Due to this cancellation, we can deduce as before from the Leibniz formula and the Moser estimates~\eqref{eq:Moser} that
\begin{align*}
\bigg| \int_{\R^N} \partial_x^\alpha \partial_j \big( \langle \partial_i m, & \partial_j m \rangle_{\R^3} \big) \big\langle \partial_t \partial_x^{\alpha^*} m, m \big\rangle_{\R^3} \bigg| \leq C_k \bigg( \| \nabla m \|_{L^\infty}^2 \| \nabla m \|_{\dot{H}^{\ell - 1}}^2 + \frac{1}{\varepsilon} \| \nabla m \|_{L^\infty} \| \nabla m \|_{\dot{H}^{\ell - 1}} \times\\
& \times \Big( \| \nabla m \|_{L^\infty} \big( \| m_1 \|_{\dot{H}^{\ell - 2}} + \| m_3 \|_{\dot{H}^{\ell - 2}} \big) + \| \nabla m \|_{\dot{H}^{\ell - 2}} \big( \| m_1 \|_{L^\infty} + \| m_3 \|_{L^\infty} \big) \Big) \bigg).
\end{align*}
We finally collect this estimate with~\eqref{missouri},~\eqref{wyoming},~\eqref{nevada} and~\eqref{connecticut} in order to bound the right-hand side of~\eqref{arkansas}. Using definition~\eqref{def:E-LL-k}, the fact that $0 < \varepsilon < 1$, and the Young inequality $2 a b \leq a^2+b^2$, we obtain~\eqref{eq:energy-estimate-LL}. This completes the proof of Proposition~\ref{prop:LL-energy-estimate}. \qed

%%%%%%%%%%%%%%%%%%%%%%%%%%%%%%%%%%%%%
%%%%%%%%%%%%%%%%%%%%%%%%%%%%%%%%%%%%%
\subsection{Proof of Corollary~\ref{cor:NLS-energy-estimate}}
\label{sub:NLS-est}
%%%%%%%%%%%%%%%%%%%%%%%%%%%%%%%%%%%%%
%%%%%%%%%%%%%%%%%%%%%%%%%%%%%%%%%%%%%

In order to establish inequality~\eqref{eq:energy-estimate-NLS}, we rewrite~\eqref{eq:energy-estimate-LL} in terms of the function $\Psi_\varepsilon$. Given an initial datum $\Psi_\varepsilon^0 \in H^{k + 5}(\R^N)$ satisfying~\eqref{cond:petit-Psi-eps}, the corresponding solution $\Psi_\varepsilon$ to~\eqref{NLS-eps} is in $\boC^0([0, T_\varepsilon), H^{k + 4}(\R^N))$ by Corollary~\ref{cor:Cauchy-Psi-eps}. Here, the characterization of the maximal time $T_\varepsilon$ guarantees that
$$\varepsilon^\frac{1}{2} \big\| \Psi_\varepsilon(\cdot, t) \big\|_{L^\infty} < 1,$$
for any $0 \leq t < T_\varepsilon$. In particular, it follows from the continuity properties of the solution $\Psi_\varepsilon$ and the Sobolev embedding theorem that
\begin{equation}
\label{uhila}
\sigma_T := \varepsilon^\frac{1}{2} \max_{t \in [0, T]} \big\| \Psi_\varepsilon(\cdot, t) \big\|_{L^\infty} < 1,
\end{equation}
for any $0 \leq T < T_\varepsilon$. 

In another direction, the function $m$ defined by~\eqref{eq:m-Psi-eps} solves~\eqref{LL} for the corresponding initial datum $m^0$, and we can prove that it is in $\boC^0([0, T], \boE^{k + 4}(\R^N))$, with $\partial_t m \in \boC^0([0, T], H^{k + 2}(\R^N))$. Indeed, the fact that $m_1$ and $m_3$ belong to $\boC^0([0, T], H^{k + 4}(\R^N))$ is a direct consequence of their definition and of the continuity properties of the function $\Psi_\varepsilon$. Concerning the function $m_2$, we write it as
\begin{equation}
\label{falgoux}
m_2 = \eta(\varepsilon^\frac{1}{2} \Psi_\varepsilon).
\end{equation}
In view of~\eqref{uhila}, the function $\eta$ in this formula can be chosen as a smooth function such that
$$\eta(x) = (1 - |x|^2)^\frac{1}{2},$$
when $|x| \leq \sigma_T$, and $\eta(x) = (1 + \sigma_T)/2$ for $|x|$ close enough to $1$. As a consequence, we have
\begin{equation}
\label{kayser}
m_2(\cdot, t) - m_2(\cdot, s) = \varepsilon^\frac{1}{2} \int_0^1 \eta' \Big( \varepsilon^\frac{1}{2} \big( (1 - \tau) \Psi_\varepsilon(\cdot, s) + \tau \Psi_\varepsilon(\cdot, t) \big) \Big) \, \big( \Psi_\varepsilon(\cdot, t) - \Psi_\varepsilon(\cdot, s) \big) \, d\tau,
\end{equation}
for any $0 \leq s \leq t \leq T$. By the Sobolev embedding theorem, we infer that
\begin{equation}
\label{vahaamahina}
\big\| m_2(\cdot, t) - m_2(\cdot, s) \big\|_{L^\infty} \leq \varepsilon^\frac{1}{2} \big\| \eta' \|_{L^\infty} \, \big\| \Psi_\varepsilon(\cdot, t) - \Psi_\varepsilon(\cdot, s) \big\|_{L^\infty} \leq \varepsilon^\frac{1}{2} \big\| \eta' \|_{L^\infty} \, \big\| \Psi_\varepsilon(\cdot, t) - \Psi_\varepsilon(\cdot, s) \big\|_{H^{k + 4}}.
\end{equation}
At this stage, let us recall the Moser estimate
\begin{equation}
\label{eq:Moser-comp}
\big\| F(f) \big\|_{\dot{H}^\ell} \leq C_\ell \| f \|_{\dot{H}^\ell} \, \max_{1 \leq m \leq \ell} \| F^{(m)} \|_{L^\infty} \, \| f \|_{L^\infty}^{m - 1},
\end{equation}
which holds for $\ell \geq 1$, $f \in L^\infty(\R^N) \cap \dot{H}^\ell(\R^N)$, and $F \in \boC^\ell(\R)$, with bounded derivatives up to order $\ell$ (see e.g.~\cite{Moser1, Hormand0}). Applying this estimate, \eqref{eq:Moser} and \eqref{uhila} to~\eqref{kayser}, we get the existence of a positive number $A$, also depending only on $k$, $\eta$ and $\sigma_T$, such that 
\begin{align*}
\big\| m_2(\cdot, t) - m_2(\cdot, s) \big\|_{\dot{H}^\ell} \leq A \varepsilon^\frac{1}{2} \Big( \big\| & \Psi_\varepsilon(\cdot, t) - \Psi_\varepsilon(\cdot, s) \big\|_{\dot{H}^\ell}\\
& + \varepsilon^\frac{1}{2} \big\| \Psi_\varepsilon(\cdot, t) - \Psi_\varepsilon(\cdot, s) \big\|_{L^\infty} \big( \big\| \Psi_\varepsilon(\cdot, t) \big\|_{\dot{H}^\ell} + \big\| \Psi_\varepsilon(\cdot, s) \big\|_{\dot{H}^\ell} \big) \Big),
\end{align*}
for any $1 \leq \ell \leq k + 4$. Combining with~\eqref{vahaamahina}, we finally deduce that $m$ is continuous on $[0, T]$, with values in $\boE^{k + 4}(\R^N)$. Since this map is solution to~\eqref{LL}, it follows from~\eqref{eq:Moser} and the Sobolev embedding theorem that its time derivative $\partial_t m$ is in $\boC^0([0, T], H^{k + 2}(\R^N))$.

As a consequence, we are in position to apply Proposition~\ref{prop:LL-energy-estimate} for the solution $m$. Coming back to~\eqref{eq:m-Psi-eps}, we express the various terms in inequality~\eqref{eq:energy-estimate-LL} as
$$\varepsilon E_{\rm LL}^\ell(t) = \gE_\varepsilon^\ell(t),$$
for any $2 \leq \ell \leq k + 2$, and
\begin{align*}
\| m_1(\cdot, t) \|_{L^\infty}^2 + \| m_3(\cdot, t) \|_{L^\infty}^2 & + \| \nabla m(\cdot, t) \|_{L^\infty}^2\\
& = \varepsilon \Big( \| \Psi_\varepsilon(\cdot, t) \|_{L^\infty}^2 + \| \nabla \Psi_\varepsilon(\cdot, t) \|_{L^\infty}^2 + \varepsilon \Big\| \frac{\langle \Psi_\varepsilon(\cdot, t), \nabla \Psi_\varepsilon(\cdot, t) \rangle_\C}{(1 - \varepsilon |\Psi_\varepsilon(\cdot, t)|^2)^\frac{1}{2}} \Big\|_{L^\infty}^2 \Big).
\end{align*}
Note also that $E_{\rm LL}^1(t) = \gE_\varepsilon^1(t)$ by~\eqref{eq:scaled-energy}. In conclusion, the continuous differentiability of the energies $\gE_\varepsilon^\ell$, as well as inequality~\eqref{eq:energy-estimate-NLS}, readily follow from Proposition~\ref{prop:LL-energy-estimate}. This completes the proof of Corollary~\ref{cor:NLS-energy-estimate}. \qed

%%%%%%%%%%%%%%%%%%%%%%%%%%%%%%%%%%%
%%%%%%%%%%%%%%%%%%%%%%%%%%%%%%%%%%%
\subsection{Proof of Lemma~\ref{lem:cont-energies-NLS}}
\label{sub:cont-energies}
%%%%%%%%%%%%%%%%%%%%%%%%%%%%%%%%%%%
%%%%%%%%%%%%%%%%%%%%%%%%%%%%%%%%%%%

We first infer from the proof of Corollary~\ref{cor:NLS-energy-estimate} that the energies $\gE_\varepsilon^\ell$ are well-defined on $[0, T]$ for $1 \leq \ell \leq k + 2$, when the solution $\Psi_\varepsilon$ to~\eqref{NLS-eps} lies in $\boC^0([0, T], H^{k + 4}(\R^N)$ and satisfies condition~\eqref{eq:cond-Psi-eps-T}. We also observe that the left-hand side inequalities in~\eqref{eq:cont-gEl-NLS} and~\eqref{eq:cont-gE1-NLS} are direct consequences of the definitions of the energies $\gE_\varepsilon^\ell$. Concerning the right-hand side inequalities, we first deal with~\eqref{eq:cont-gE1-NLS}, for which condition~\eqref{eq:cond-Psi-eps-T} guarantees that
$$\bigg| \frac{\varepsilon^2 \langle \Psi_\varepsilon(x, t), \nabla \Psi_\varepsilon(x, t) \rangle_\C^2}{1 - \varepsilon |\Psi_\varepsilon(x, t)|^2} \bigg| \leq \frac{\varepsilon \sigma_T^2 |\nabla \Psi_\varepsilon(x, t)|^2}{1 - \sigma_T^2},$$
for any $(x, t) \in \R^N \times [0, T]$. The inequality then follows, with $C = 1/(2 - 2 \sigma_T^2)$.

We argue similarly for the right-hand side inequality in~\eqref{eq:cont-gEl-NLS}. We take advantage of the uniform bound given by~\eqref{eq:cond-Psi-eps-T} in order to control the space derivatives of the function
$(1 - \varepsilon |\Psi_\varepsilon|^2)^{1/2}$. As in the proof of Corollary~\ref{cor:NLS-energy-estimate}, we introduce a smooth function such that
$$\eta(x) = (1 - |x|^2)^\frac{1}{2},$$
when $|x| \leq \sigma_T$, and $\eta(x) = (1 + \sigma_T)/2$ for $|x|$ close enough to $1$. Since $(1 - \varepsilon |\Psi_\varepsilon|^2)^{1/2} = \eta(\varepsilon^\frac{1}{2} \Psi_\varepsilon)$ by~\eqref{eq:cond-Psi-eps-T}, we again deduce from the Moser estimate~\eqref{eq:Moser-comp} that
$$\big\| \Delta (1 - \varepsilon |\Psi_\varepsilon(\cdot, t)|^2)^\frac{1}{2} \big\|_{\dot{H}^{\ell - 2}} \leq C \varepsilon^\frac{1}{2} \big\| \Delta \Psi_\varepsilon(\cdot, t) \big\|_{\dot{H}^{\ell - 2}} \max_{1 \leq m \leq \ell} \varepsilon^\frac{m - 1}{2} \big\| \Psi_\varepsilon(\cdot, t) \big\|_{L^\infty}^{m - 1},$$
where $C$ refers, here as in the sequel, to a positive number depending only on $k$ and $\sigma_T$ (once the choice of the function $\eta$ is fixed). Condition~\eqref{eq:cond-Psi-eps-T} then provides
\begin{equation}
\label{fidelma}
\varepsilon \big\| \Delta (1 - \varepsilon |\Psi_\varepsilon(\cdot, t)|^2)^\frac{1}{2} \big\|_{\dot{H}^{\ell - 2}}^2 \leq C^2 \varepsilon^2 \big\| \Delta \Psi_\varepsilon(\cdot, t) \big\|_{\dot{H}^{\ell - 2}}^2.
\end{equation}

At this stage, we are left with the estimates of the two terms in~\eqref{def:gE-NLS}, which depend on the time derivative of the function $\Psi_\varepsilon$. Concerning the first one, we rewrite~\eqref{NLS-eps} as
$$i \varepsilon \partial_t \Psi_\varepsilon + \varepsilon \div \Big( \big( 1 - \varepsilon |\Psi_\varepsilon|^2 \big)^\frac{1}{2} \nabla \Psi_\varepsilon - \nabla \big( 1 - \varepsilon |\Psi_\varepsilon|^2 \big)^\frac{1}{2} \Psi_\varepsilon \Big) + \big( 1 - (1 - \varepsilon |\Psi_\varepsilon|^2)^\frac{1}{2} \big) \Psi_\varepsilon = 0,$$
so that
\begin{equation}
 \label{eadulf}
\begin{split}
\big\| \varepsilon \partial_t \Psi_\varepsilon(\cdot, t) & - i \Psi_\varepsilon(\cdot, t) \big\|_{\dot{H}^{\ell - 2}} \leq \varepsilon \big\| \big( 1 - \varepsilon |\Psi_\varepsilon(\cdot, t)|^2 \big)^\frac{1}{2} \nabla \Psi_\varepsilon(\cdot, t) \big\|_{\dot{H}^{\ell - 1}}\\
& + \varepsilon \big\| \nabla \big( 1 - \varepsilon |\Psi_\varepsilon(\cdot, t)|^2 \big)^\frac{1}{2} \Psi_\varepsilon(\cdot, t) \big\|_{\dot{H}^{\ell - 1}} + \big\| \big( 1 - \varepsilon |\Psi_\varepsilon(\cdot, t)|^2 \big)^\frac{1}{2} \Psi_\varepsilon(\cdot, t) \big\|_{\dot{H}^{\ell - 2}}.
\end{split}
\end{equation}
Invoking the Moser estimate~\eqref{eq:Moser}, we have
\begin{align*}
\big\| \big( 1 - \varepsilon |\Psi_\varepsilon(\cdot, t)|^2 \big)^\frac{1}{2} \nabla \Psi_\varepsilon(\cdot, t) \big\|_{\dot{H}^{\ell - 1}} & + \big\| \nabla \big( 1 - \varepsilon |\Psi_\varepsilon(\cdot, t)|^2 \big)^\frac{1}{2} \Psi_\varepsilon(\cdot, t) \big\|_{\dot{H}^{\ell - 1}}\\
& \leq C \Big( \big\| \Psi_\varepsilon(\cdot, t) \big\|_{\dot{H}^\ell} + \big\| \Psi_\varepsilon(\cdot, t) \big\|_{L^\infty} \big\| \big( 1 - \varepsilon |\Psi_\varepsilon(\cdot, t)|^2 \big)^\frac{1}{2} \big\|_{\dot{H}^\ell} \Big),
\end{align*}
so that~\eqref{eq:cond-Psi-eps-T} and~\eqref{fidelma} provide
$$\big\| \big( 1 - \varepsilon |\Psi_\varepsilon(\cdot, t)|^2 \big)^\frac{1}{2} \nabla \Psi_\varepsilon(\cdot, t) \big\|_{\dot{H}^{\ell - 1}} + \big\| \nabla \big( 1 - \varepsilon |\Psi_\varepsilon(\cdot, t)|^2 \big)^\frac{1}{2} \Psi_\varepsilon(\cdot, t) \big\|_{\dot{H}^{\ell - 1}} \leq C \big\| \Delta \Psi_\varepsilon(\cdot, t) \big\|_{\dot{H}^{\ell - 2}}.$$
Similarly, we can bound the last term in the right-hand side of~\eqref{eadulf} by $\| \Psi_\varepsilon \|_{\dot{H}^{\ell - 2}}$ if $\ell = 2$, and by
$$\big\| \big( 1 - \varepsilon |\Psi_\varepsilon(\cdot, t)|^2 \big)^\frac{1}{2} \Psi_\varepsilon(\cdot, t) \big\|_{\dot{H}^{\ell - 2}} \leq C \Big( \| \Psi_\varepsilon(\cdot, t) \|_{\dot{H}^{\ell - 2}} + \big\| \Psi_\varepsilon(\cdot, t) \big\|_{L^\infty} \big\| \big( 1 - \varepsilon |\Psi_\varepsilon(\cdot, t)|^2 \big)^\frac{1}{2} \big\|_{\dot{H}^{\ell - 2}} \Big),$$
otherwise. Since $\ell - 2 \geq 1$ in this case, we again infer from~\eqref{eq:cond-Psi-eps-T} and~\eqref{eq:Moser-comp} that
$$\big\| \Psi_\varepsilon(\cdot, t) \big\|_{L^\infty} \big\| \big( 1 - \varepsilon |\Psi_\varepsilon(\cdot, t)|^2 \big)^\frac{1}{2} \big\|_{\dot{H}^{\ell - 2}} \leq C \big\| \Psi_\varepsilon(\cdot, t) \big\|_{\dot{H}^{\ell - 2}}.$$
Gathering the previous estimates of the right-hand side of~\eqref{eadulf}, we finally get
\begin{equation}
\label{colgu}
\big\| \varepsilon \partial_t \Psi_\varepsilon(\cdot, t) - i \Psi_\varepsilon(\cdot, t) \big\|_{\dot{H}^{\ell - 2}}^2 \leq C \Big( \varepsilon^2 \big\| \Delta \Psi_\varepsilon(\cdot, t) \big\|_{\dot{H}^{\ell - 2}}^2 + \big\| \Psi_\varepsilon(\cdot, t) \big\|_{\dot{H}^{\ell - 2}}^2 \Big).
\end{equation}

We now turn to the last term in~\eqref{def:gE-NLS}. Coming back to~\eqref{NLS-eps}, we infer that
$$\partial_t \big( 1 - \varepsilon |\Psi_\varepsilon|^2 \big)^\frac{1}{2} = \varepsilon \div \langle i \Psi_\varepsilon, \nabla \Psi_\varepsilon \rangle_\C,$$
and we again deduce from~\eqref{eq:cond-Psi-eps-T} and~\eqref{eq:Moser} that
$$\varepsilon \big\| \partial_t \big( 1 - \varepsilon |\Psi_\varepsilon(\cdot, t)|^2 \big)^\frac{1}{2} \big\|_{\dot{H}^{\ell - 2}}^2 \leq \varepsilon^3 \big\| \Psi_\varepsilon(\cdot, t) \big\|_{L^\infty}^2 \big\| \Delta \Psi_\varepsilon(\cdot, t) \big\|_{\dot{H}^{\ell - 2}}^2 \leq C \varepsilon^2 \big\| \Delta \Psi_\varepsilon(\cdot, t) \big\|_{\dot{H}^{\ell - 2}}^2.$$
Combining this inequality with~\eqref{eadulf} and~\eqref{colgu}, we conclude that the energy $\gE_\varepsilon^\ell(t)$ can be bounded from above according to~\eqref{eq:cont-gEl-NLS}. This completes the proof of Lemma~\ref{lem:cont-energies-NLS}. \qed

%%%%%%%%%%%%%%%%%%%%%%%%%%%%%%%%%
%%%%%%%%%%%%%%%%%%%%%%%%%%%%%%%%%
\subsection{Proof of Proposition~\ref{prop:estim-eps}}
\label{sub:estim-eps}
%%%%%%%%%%%%%%%%%%%%%%%%%%%%%%%%%
%%%%%%%%%%%%%%%%%%%%%%%%%%%%%%%%%

The proof relies on a continuation argument, which is based on the Sobolev control of the solution $\Psi_\varepsilon$ provided by the energies $\gE_\varepsilon^k$. Assume first that the initial condition $\Psi_\varepsilon^0$ lies in $H^{k + 5}(\R^N)$ and satisfies condition~\eqref{cond:Psi-eps-petit-0} for a positive number $C_*$ to be fixed later. In this case, Corollary~\ref{cor:Cauchy-Psi-eps} yields the existence of a maximal time $T_{\max}$ and of a unique solution $\Psi_\varepsilon \in \boC^0([0, T_{\max}), H^{k + 4}(\R^N))$ to~\eqref{NLS-eps} with initial datum $\Psi_\varepsilon^0$. In particular, we infer from Corollary~\ref{cor:NLS-energy-estimate} that the quantity $\Sigma_\varepsilon^{k + 2}$ defined by
$$\Sigma_\varepsilon^{k + 2} := \sum_{\ell = 1}^{k + 2} \gE_\varepsilon^\ell,$$ 
is well-defined and of class $\boC^1$ on $[0, T_{\max})$. On the other hand, we can invoke the Sobolev embedding theorem so as to find a positive number $C_1$, depending only on $k$, such that 
$$\varepsilon^\frac{1}{2} \big\| \Psi_\varepsilon^0 \big\|_{L^\infty} \leq \varepsilon^\frac{1}{2} C_1 \big\| \Psi_\varepsilon ^0 \big\|_{H^{k - 1}}.$$
Assuming that condition~\eqref{cond:Psi-eps-petit-0} holds for a number $C_*$ such that $\sigma C_* \geq 2 C_1$, we obtain
$$\varepsilon^\frac{1}{2} \big\| \Psi_\varepsilon^0 \big\|_{L^\infty}\leq \frac{C_1}{C_*} \leq \frac{\sigma}{2}.$$
As a consequence of the continuity properties of the quantity $\Sigma_\varepsilon^{k + 2}$ and of the solution $\Psi_\varepsilon$, we deduce that the stopping time 
\begin{equation}
\label{def:T*}
T_* := \sup \Big\{ t \in [0, T_{\max}) : \varepsilon^\frac{1}{2} \big\| \Psi_\varepsilon(\cdot, \tau) \big\|_{L^\infty}\leq \sigma \text{ and } \Sigma_\varepsilon^{k + 2}(\tau) \leq 2 \Sigma_\varepsilon^{k + 2}(0) \ {\rm for} \ {\rm any} \ \tau \in [0, t] \Big\},
\end{equation}
is positive.

At this stage, we go back to inequality~\eqref{eq:energy-estimate-NLS} in order to find a further positive number $C_2$, also depending only on $k$, such that
$$\big[ \Sigma_\varepsilon^{k + 2} \big]'(t) \leq C_2 \, \bigg( \| \Psi_\varepsilon(\cdot, t) \|_{L^\infty}^2 + \| \nabla \Psi_\varepsilon(\cdot, t) \|_{L^\infty}^2 + \varepsilon \Big\| \frac{\langle \Psi_\varepsilon(\cdot, t), \nabla \Psi_\varepsilon(\cdot, t) \rangle_\C}{(1 - \varepsilon |\Psi_\varepsilon(\cdot, t)|^2)^\frac{1}{2}} \Big\|_{L^\infty}^2 \bigg) \, \Sigma_\varepsilon^{k + 2}(t),$$
for any $0 \leq t < T_*$. Due to the definition of $T_*$, we observe that
$$\varepsilon \Big\| \frac{\langle \Psi_\varepsilon(\cdot, t), \nabla \Psi_\varepsilon(\cdot, t) \rangle_\C}{(1 - \varepsilon |\Psi_\varepsilon(\cdot, t)|^2)^\frac{1}{2}} \Big\|_{L^\infty}^2 \leq \frac{1}{1 - \sigma^2} \| \nabla \Psi_\varepsilon(\cdot, t) \|_{L^\infty}^2.$$
Combining these inequalities with the Sobolev embedding theorem, we are led to
$$\big[ \Sigma_\varepsilon^{k + 2} \big]'(t) \leq \frac{C_1^2 C_2}{1 - \sigma^2} \, \Sigma_\varepsilon^{k + 2}(t)^2.$$
Setting
\begin{equation}
\label{def:T-eps}
T_\varepsilon := \frac{1 - \sigma^2}{2 C_1^2 C_2 \Sigma_\varepsilon^{k + 2}(0)},
\end{equation}
and integrating the previous inequality, we infer that
\begin{equation}
\label{aichu}
\Sigma_\varepsilon^{k + 2}(t) \leq \frac{(1 - \sigma^2) \Sigma_\varepsilon^{k + 2}(0)}{1 - \sigma^2 - C_1^2 C_2 \Sigma_\varepsilon^{k + 2}(0) t} \leq 2 \Sigma_\varepsilon^{k + 2}(0),
\end{equation}
for any $t < T_\varepsilon$. Invoking once again the Sobolev embedding theorem and the definition of the quantity $\Sigma_\varepsilon^{k + 2}$, we also get
$$\varepsilon^\frac{1}{2} \| \Psi_\varepsilon(\cdot, t) \|_{L^\infty} \leq C_1 \varepsilon^\frac{1}{2} \Sigma_\varepsilon^{k + 2}(t)^\frac{1}{2} \leq \sqrt{2} C_1 \varepsilon^\frac{1}{2} \Sigma_\varepsilon^{k + 2}(0)^\frac{1}{2}.$$
In view of Lemma~\ref{lem:cont-energies-NLS}, we infer the existence of a positive number $C_3$, depending only on $k$ and $\sigma$, such that
$$\varepsilon^\frac{1}{2} \| \Psi_\varepsilon(\cdot, t) \|_{L^\infty} \leq \sqrt{2} C_1 C_3 \varepsilon^\frac{1}{2} \Big( \big\| \Psi_\varepsilon^0 \big\|_{H^k} + \varepsilon^\frac{1}{2} \big\| \nabla \Psi_\varepsilon^0 \big\|_{\dot{H}^k} + \varepsilon \big\| \Delta \Psi_\varepsilon^0 \big\|_{\dot{H}^k} \Big),$$
again when $t < T_\varepsilon$. Enlarging $C_*$ so that $\sqrt{2} C_1 C_3 \leq \sigma C_*$, we deduce that
\begin{equation}
\label{finguine}
\varepsilon^\frac{1}{2} \| \Psi_\varepsilon(\cdot, t) \|_{L^\infty} \leq \sigma.
\end{equation}
In view of~\eqref{aichu}, a continuation argument then guarantees that either $T_\varepsilon \leq T_* \leq T_{\max}$, or $T_* = T_{\max} < T_\varepsilon$. In this latter case, it results from the conditions in~\eqref{eq:cond-Tmax-NLS-eps} and from~\eqref{finguine} that
\begin{equation}
\label{enda}
\int_0^{T_{\max}} \big\| \nabla \Psi_\varepsilon(\cdot, t) \big\|_{L^\infty}^2 \, dt = \infty.
\end{equation}
On the other hand, as a further consequence of the Sobolev embedding theorem, of Lemma~\ref{lem:cont-energies-NLS} and of~\eqref{aichu}, we have
$$\int_0^{T_{\max}} \big\| \nabla \Psi_\varepsilon(\cdot, t) \big\|_{L^\infty}^2 \, dt \leq C_1^2 \int_0^{T_{\max}} \big\| \Psi_\varepsilon(\cdot, t) \big\|_{H^k}^2 \, dt \leq 4 C_1^2 \Sigma_\varepsilon^{k + 2}(0) T_{\max}.$$
When $T_{\max} < T_\varepsilon$, definition~\eqref{def:T-eps} yields
$$\int_0^{T_{\max}} \big\| \nabla \Psi_\varepsilon(\cdot, t) \big\|_{L^\infty}^2 \, dt \leq \frac{2 - 2 \sigma^2}{C_2} < \infty,$$
which contradicts~\eqref{enda}. As a conclusion, the stopping time $T_*$ is at least equal to $T_\varepsilon$, and we derive from Lemma~\ref{lem:cont-energies-NLS} and from~\eqref{aichu} that
\begin{align*}
\big\| \Psi_\varepsilon(\cdot, t) \big\|_{\dot{H}^k}^2 & + \varepsilon \big\| \nabla \Psi_\varepsilon(\cdot, t) \big\|_{\dot{H}^k}^2 + \varepsilon^2 \big\| \Delta \Psi_\varepsilon(\cdot, t) \big\|_{\dot{H}^k}^2\\
& \leq 2 \Sigma_\varepsilon^{k + 2}(t) \leq 4 \Sigma_\varepsilon^{k + 2}(0) \leq 4 C_3^2 \Big( \big\| \Psi_\varepsilon^0 \big\|_{\dot{H}^k}^2 + \varepsilon \big\| \nabla \Psi_\varepsilon^0 \big\|_{\dot{H}^k}^2 + \varepsilon^2 \big\| \Delta \Psi_\varepsilon^0 \big\|_{\dot{H}^k}^2 \Big),
\end{align*}
for any $0 \leq t \leq T_\varepsilon$. Similarly, we derive from~\eqref{def:T-eps} that
$$T_\varepsilon \geq \frac{1}{2 C_1^2 C_2 C_3^2 \big( \| \Psi_\varepsilon^0 \|_{\dot{H}^k}^2 + \varepsilon \| \nabla \Psi_\varepsilon^0 \|_{\dot{H}^k}^2 + \varepsilon^2 \| \Delta \Psi_\varepsilon^0 \|_{\dot{H}^k}^2 \big)}.$$
It then remains to again enlarge the number $C_*$ so that $C_* \geq 4 C_3^2$ and $C_*\geq 2 C_1^2 C_2 C_3^2$ in order to complete the proof of Proposition~\ref{prop:estim-eps}, provided that $\Psi_\varepsilon^0 \in H^{k + 5}(\R^N)$. In view of the continuity of the~\eqref{NLS-eps} flow with respect to the initial datum in Corollary~\ref{cor:Cauchy-Psi-eps}, we can extend this result to arbitrary initial conditions $\Psi_\varepsilon^0 \in H^{k + 3}(\R^N)$ by a standard density argument. This ends the proof of Proposition~\ref{prop:estim-eps}. \qed

%%%%%%%%%%%%%%%%%%%%%%%%%%%%%%
%%%%%%%%%%%%%%%%%%%%%%%%%%%%%%
\subsection{Proof of Lemma~\ref{lem:estim-CS}}
\label{sub:estim-CS}
%%%%%%%%%%%%%%%%%%%%%%%%%%%%%%
%%%%%%%%%%%%%%%%%%%%%%%%%%%%%%

Assume first that the initial condition $\Psi^0$ belongs to $H^{k + 2}(\R^N)$ and consider the corresponding solution $\Psi \in \boC^0([0, T_{\max}), H^{k + 2}(\R^N))$ to~\eqref{CS}. In this case, the derivative $\partial_t \Psi$ is in $\boC^0([0, T_{\max}), H^k(\R^N))$, so that the quantity
$$\gS_k(t) = \frac{1}{2} \big\| \Psi(\cdot, t) \big\|_{H^k}^2,$$
is well-defined and of class $\boC^1$ on the interval $[0, T_{\max})$. In view of~\eqref{CS}, its derivative is equal to
$$\gS'_k(t) = \Big\langle - \Delta \Psi(\cdot, t) - \frac{1}{2} |\Psi(\cdot, t)|^2 \Psi(\cdot, t), i \Psi(\cdot, t) \Big\rangle_{H^k},$$
for any $0 \leq t \leq T_{\max}$. Since we have
$$- \big\langle \Delta \Psi(\cdot, t), i \Psi(\cdot, t) \big\rangle_{H^k} = \big\langle \nabla \Psi(\cdot, t), i \nabla \Psi(\cdot, t) \big\rangle_{H^k} = 0,$$
by integration by parts, we can combine the Moser estimates in~\eqref{eq:Moser} and the Sobolev embedding theorem in order to get
$$\gS'_k(t) \leq \frac{1}{2} \big\| |\Psi(\cdot, t)|^2 \Psi(\cdot, t) \big\|_{H^k} \big\| \Psi(\cdot, t) \big\|_{H^k} \leq C_k \big\| \Psi(\cdot, t) \big\|_{L^\infty}^2 \big\| \Psi(\cdot, t) \big\|_{H^k}^2 \leq C_k \gS_k(t)^2.$$
Here, the positive number $C_k$, possibly changing from inequality to inequality, only depends on $k$. Setting
\begin{equation}
\label{caol}
T_* := \frac{1}{2 C_k \gS_k(0)} = \frac{1}{C_k \| \Psi^0 \big\|_{H^k}^2},
\end{equation}
we conclude that
\begin{equation}
\label{gorman}
\gS_k(t) \leq \frac{\gS_k(0)}{1 - C_k t \gS_k(0)} \leq 2 \gS_k(0) = \| \Psi^0 \big\|_{H^k}^2,
\end{equation}
as long as $t \leq T_*$. Invoking again the Sobolev embedding theorem, we also have
$$\big\| \Psi(\cdot, t) \big\|_{L^\infty} \leq C_k \gS_k(t)^\frac{1}{2} \leq C_k \| \Psi^0 \big\|_{H^k} < \infty,$$
and a continuation argument as in the proof of Proposition~\ref{prop:estim-eps} guarantees that the maximal time of existence $T_{\max}$ is greater than $T_*$. In view of~\eqref{caol} and~\eqref{gorman}, this completes the proof of Lemma~\ref{lem:estim-CS} when $\Psi^0 \in H^{k + 2}(\R^N)$. Using the uniform lower bound on the maximal time of existence provided by $T_*$, we can finally perform a standard density argument to extend this lemma to any arbitrary initial datum $\Psi^0 \in H^k(\R^N)$. This concludes the proof of Lemma~\ref{lem:estim-CS}. \qed

%%%%%%%%%%%%%%%%%%%%%%%%%%%%%%
%%%%%%%%%%%%%%%%%%%%%%%%%%%%%%
\subsection{Proof of Proposition~\ref{prop:error}}
\label{sub:error}
%%%%%%%%%%%%%%%%%%%%%%%%%%%%%%
%%%%%%%%%%%%%%%%%%%%%%%%%%%%%%

Set
$$\gS_{k - 2}(t) = \frac{1}{2} \big\| u_\varepsilon(\cdot, t) \big\|_{H^{k - 2}}^2,$$
for any $t \in [0, T]$. Under the assumptions of Proposition~\ref{prop:error}, the functions $u_\varepsilon$ and $\partial_t u_\varepsilon$ lie in $\boC^0([0, T], H^k)$ and in $\boC^0([0, T], H^{k - 2})$, respectively. Hence, the function $\gS_{k - 2}$ is of class $\boC^1$ on $[0, T]$. In view of~\eqref{eq:diff}, its derivative is given by
$$\gS'_{k - 2}(t) = \big\langle - \Delta u_\varepsilon(\cdot, t) - G_\varepsilon(\cdot, t) + \varepsilon \boR_\varepsilon(\cdot, t), i u_\varepsilon(\cdot, t) \big\rangle_{H^{k - 2}},$$
where we denote
$$G_\varepsilon := \frac{1}{2} \big( |u_\varepsilon + \Psi|^2 (u_\varepsilon + \Psi) - |\Psi|^2 \Psi \big).$$
Integrating by parts, we see that
$$\big\langle \Delta u_\varepsilon, i u_\varepsilon \big\rangle_{H^{k - 2}} = - \big\langle \nabla u_\varepsilon, i \nabla u_\varepsilon \big\rangle_{H^{k - 2}} = 0,$$
so that the Cauchy-Schwarz inequality leads to
\begin{equation}
\label{diff:S}
\big| \gS'_{k - 2}(t) \big| \leq \sqrt2\gS_{k - 2}(t)^\frac{1}{2} \Big( \big\| G_\varepsilon(\cdot, t) \big\|_{H^{k - 2}} + \varepsilon \big\| \boR_\varepsilon(\cdot, t) \big\|_{H^{k - 2}} \Big).
\end{equation}

At this point, we notice that
$$G_\varepsilon = \frac{1}{2} \big( |u_\varepsilon|^2 u_\varepsilon + |u_\varepsilon|^2 \Psi + 2 \langle u_\varepsilon, \Psi\rangle_\C (u_\varepsilon + \Psi) + |\Psi|^2 u_\varepsilon \big),$$
and we invoke the Moser estimates~\eqref{eq:Moser}, as well as the Young inequality $2 a b \leq a^2+b^2$, in order to obtain
$$\big\| G_\varepsilon \big\|_{H^{k - 2}} \leq C_k \Big( \| u_\varepsilon \|_{H^{k - 2}} \, \big( \| u_\varepsilon \|_{L^\infty}^2 + \| \Psi \|_{L^\infty}^2 \big) + \| \Psi \|_{H^{k - 2}} \, \| u_\varepsilon \|_{L^\infty} \, \big( \| u_\varepsilon \|_{L^\infty} + \| \Psi \|_{L^\infty} \big) \Big),$$
Here as in the sequel, the notation $C_k$ refers to a positive number depending only on $k$. Due to the assumption $k > N/2 + 2$, the Sobolev embedding of $H^{k - 2}(\R^N)$ into $L^\infty(\R^N)$ then leads to 
\begin{equation}
\label{est-G-m}
\big\| G_\varepsilon(\cdot, t) \big\|_{H^{k - 2}} \leq C_k \, \gS_{k - 2}(t)^\frac{1}{2} \, \boK_\varepsilon(T)^2. 
\end{equation}

We now turn to the remainder term $\boR_\varepsilon$, which we decompose as $\boR_\varepsilon := \boR_{\varepsilon, 1} - \boR_{\varepsilon, 2} - \boR_{\varepsilon, 3}$ according to the three terms in the right-hand side of~\eqref{def:R-eps}. We introduce a smooth function $\chi$ such that 
$$\chi(x) = \frac{1}{1 + (1 - |x|^2)^\frac{1}{2}},$$
when $|x| \leq \sigma$, and we use~\eqref{eq:unif-bound} to recast $\boR_{\varepsilon, 1}$ as 
$$\boR_{\varepsilon, 1} = |\Psi_\varepsilon|^2 \, \Delta \Psi_\varepsilon \, \chi \big( \varepsilon^\frac{1}{2} \, \Psi_\varepsilon \big).$$
We next apply the Moser estimate \eqref{eq:Moser-comp} and invoke again~\eqref{eq:unif-bound} to obtain
$$\big\| \chi (\varepsilon^\frac{1}{2} \, \Psi_\varepsilon) \big\|_{\dot{H}^\ell} \leq C_{k, \sigma} \| \Psi_\varepsilon \|_{\dot{H}^\ell},$$
when $1 \leq \ell \leq k - 2$. Here as in the sequel, the notation $C_{k, \sigma}$ refers to a positive number depending only on $k$ and $\sigma$. In view of~\eqref{eq:Moser}, this gives
$$\big\| \boR_{\varepsilon, 1} \big\|_{\dot{H}^\ell} \leq C_{k, \sigma} \| \Psi_\varepsilon \|_{L^\infty} \Big( \| \Delta \Psi_\varepsilon \|_{L^\infty} \big( 1 + \| \Psi_\varepsilon \|_{L^\infty} \big) \| \Psi_\varepsilon \|_{\dot{H}^\ell} + \| \Psi_\varepsilon \|_{L^\infty} \| \Psi_\varepsilon \|_{\dot{H}^{\ell + 2}} \Big).$$
Since
$$\big\| \boR_{\varepsilon, 1} \big\|_{L^2} \leq \| \Psi_\varepsilon \|_{L^\infty}^2 \| \Psi_\varepsilon \|_{\dot{H}^2},$$
due to~\eqref{eq:unif-bound}, we deduce from the condition $k > N/2 + 2$ and the Sobolev embedding theorem that
$$\big\| \boR_{\varepsilon, 1} \big\|_{H^{k - 2}} \leq C_{k, \sigma} \boK_\varepsilon(T)^3 \Big( 1 + \boK_\varepsilon(T) \Big).$$
We argue similarly for the term
$$\boR_{\varepsilon, 2} = \frac{1}{2} |\Psi_\varepsilon|^4 \, \chi \big( \varepsilon^\frac{1}{2} \, \Psi_\varepsilon \big)^2 \, \Psi_\varepsilon,$$
which we bound as
$$\big\| \boR_{\varepsilon, 2} \big\|_{H^{k - 2}} \leq C_{k, \sigma} \boK_\varepsilon(T)^5 \Big( 1 + \boK_\varepsilon(T) \Big).$$
Concerning the term $\boR_{\varepsilon, 3}$, we introduce a further smooth function $\rho$ such that
$$\rho(x) = \frac{1}{(1 - |x|^2)^\frac{1}{2}},$$
when $|x| \leq \sigma$, and we write
$$\boR_{\varepsilon, 3} = \Big( \big( \langle \Psi_\varepsilon, \Delta \Psi_\varepsilon \rangle_\C + |\nabla \Psi_\varepsilon|^2 \big) \, \rho \big( \varepsilon^\frac{1}{2} \, \Psi_\varepsilon \big) + \varepsilon^\frac{1}{2} \, \langle \Psi_\varepsilon, \nabla \Psi_\varepsilon \rangle_\C \cdot \nabla \Psi_\varepsilon \, \rho' \big( \varepsilon^\frac{1}{2} \, \Psi_\varepsilon \big) \Big) \Psi_\varepsilon.$$
Arguing as for the function $\chi$, we have
$$\big\| \rho(\varepsilon^\frac{1}{2} \, \Psi_\varepsilon) \big\|_{\dot{H}^\ell} + \big\| \rho'(\varepsilon^\frac{1}{2} \, \Psi_\varepsilon) \big\|_{\dot{H}^\ell} \leq C_{k, \sigma} \| \Psi_\varepsilon \|_{\dot{H}^\ell},$$
for $1 \leq \ell \leq k - 2$, and we infer as before that
$$\big\| \boR_{\varepsilon, 3} \big\|_{H^{k - 2}} \leq C_{k, \sigma} \boK_\varepsilon(T)^3 \Big( 1 + \boK_\varepsilon(T) \Big).$$
We conclude that
$$\big\| \boR_\varepsilon \big\|_{H^{k - 2}} \leq C_{k, \sigma} \boK_\varepsilon(T)^3 \Big( 1 + \boK_\varepsilon(T)^3 \Big).$$
Coming back to~\eqref{diff:S} and using~\eqref{est-G-m}, as well as the Young inequality, we obtain
$$\gS_{k - 2}'(t) \leq C_{k, \sigma} \boK_\varepsilon(T)^2 \Big( \gS_{k - 2}(t) + \varepsilon^2 \boK_\varepsilon(T)^2 \big( 1 + \boK_\varepsilon(T)^6 \big) \Big).$$
Estimate~\eqref{est:error} finally follows from the Gronwall inequality. \qed

%%%%%%%%%%%%%%%%%%%%%%%%%%%%%%%%%%%%%%%%%%%%%%%%%%
%%%%%%%%%%%%%%%%%%%%%%%%%%%%%%%%%%%%%%%%%%%%%%%%%%
%%%%%%%%%%%%%%%%%%%%%%%%%%%%%%%%%%%%%%%%%%%%%%%%%%
\appendix
\section{Solitons of the Landau-Lifshitz equation}
\label{sec:solitons}
%%%%%%%%%%%%%%%%%%%%%%%%%%%%%%%%%%%%%%%%%%%%%%%%%%
%%%%%%%%%%%%%%%%%%%%%%%%%%%%%%%%%%%%%%%%%%%%%%%%%%
%%%%%%%%%%%%%%%%%%%%%%%%%%%%%%%%%%%%%%%%%%%%%%%%%%

In this appendix, we focus on the correspondence between the solitons of the one-dimensional Landau-Lifshitz equation and of the one-dimensional cubic Schr\"odinger equation. Concerning this latter equation, it is well-known that it owns bright solitons (see e.g.~\cite{SuleSul0}). Up to a space translation and a phase shift, they are given by 
\begin{equation}
\label{sol:NLS}
\Psi_{c, \omega}(x, t) = \frac{(4 \omega - c^2)^\frac{1}{2} \, e^{i \frac{c (x - c t)}{2}}}{\cosh \big( \frac{ (4 \omega - c^2)^\frac{1}{2}}{2} (x - c t) \big)} \, e^{i \omega t},
\end{equation}
for any $(x, t) \in \R^2$. In this formula, the speed $c \in \R$ and the angular velocity $\omega \in \R$ satisfy the condition $4 \omega > c^2$. In the sequel, our goal is to exhibit solitons for the one-dimensional Landau-Lifshitz equation, which converge towards the bright solitons $\Psi_{c, \omega}$ in the cubic Schr\"odinger regime that we have derived in Theorem~\ref{thm:conv-CLS}. 

Going back to the scaling in~\eqref{def:Psi-eps} and to formula~\eqref{sol:NLS}, we look for solitons to~\eqref{LL} under the form
\begin{equation}
\label{ansatz}
\check m_{c, \omega}(x, t) = \check{V}_{c, \omega}(x - c t) e^{i \omega t} \quad \text{and} \quad [m_{c, \omega}]_2(x,t) := [V_{c, \omega}]_2(x - c t),
\end{equation}
for any $(x, t) \in \R^2$. Here, we have set, as before, $\check{V}_{c, \omega} = [V_{c, \omega}]_1 + i [V_{c, \omega}]_3$. The speed $c$ and the angular momentum $\omega$ are real numbers. In order to simplify the analysis, we also assume that $\lambda := \lambda_1 = \lambda_3 > 0$ as in the cubic Schr\"odinger regime. Using the equivalent formulation of the one-dimensional Landau-Lifshitz equation given by
$$m \times \partial_t m - \partial_{xx} m - \big( |\partial_x m|^2 + \lambda (m_1^2 + m_3^2) \big) m + \lambda \big( m_1 e_1 + m_3 e_3 \big) = 0,$$
we observe that the functions $\check{V}_{c, \omega}$ and $[V_{c, \omega}]_2$ solve the ordinary differential system
\begin{equation}
\label{TW}
\tag{TW$_{c, \omega}$}
\begin{cases} - \check{v}'' + i c \big( v_2 \check{v}' - v_2' \check{v} \big) - \big( |\check{v}'|^2 + |v_2'|^2 + \lambda |\check{v}|^2 \big) \check{v} + \lambda \check{v} + \omega v_2 \check{v} = 0,\\
- v_2'' + c \langle i \check{v}, \check{v}' \rangle_\C - \big( |\check{v}'|^2 + |v_2'|^2 + \lambda |\check{v}|^2 \big) v_2 - \omega |\check{v}|^2 = 0.
\end{cases}
\end{equation}
This system appears as a perturbation of the harmonic maps equation, which corresponds to the case $c = \lambda = \omega = 0$. It is invariant by translations and by phase shifts (of the function $\check{v}$). In the energy space $\boE(\R)$, the unique constant solutions to~\eqref{TW} are the trivial solutions $e_2 = (0, 1, 0)$ and $- e_2 = (0, - 1, 0)$. Moreover, we are able to classify all the non-trivial solutions in this space according to the possible values of the parameters $c$, $\lambda$ and $\omega$. 
 
\begin{thm} 
\label{thm:soliton}
Let $\lambda > 0$ and $(c, \omega) \in \R^2$. Up to the invariance by translations and phase shifts (of the map $\check{V}_{c, \omega}$), the unique non-trivial solutions $V_{c, \omega}$ to~\eqref{TW} in the energy space $\boE(\R)$ are given by the following formulae :

\noindent $(i)$ For $\omega = c = 0$, 
$$\forall x \in \R, \check{V}_{0, 0}(x) = \frac{1}{\cosh(\lambda^\frac{1}{2} x)}, \quad \text{and} \quad [V_{0, 0}]_2(x) = \delta \tanh \big( \lambda^\frac{1}{2} x \big),$$
with $\delta \in \{ \pm 1 \}$.

\noindent $(ii)$ For $0 < - \omega \delta < \lambda$ and $c^2 < 4 \big( \lambda + \omega \delta \big)$, or for $\omega \delta \geq 0$ and $0 < c^2 < 4 \big( \lambda + \omega \delta \big)$,
\begin{align*}
\forall x \in \R, \check{V}_{c, \omega}(x) = & \frac{(4 (\lambda + \delta \omega) - c^2)^\frac{1}{2} \, e^\frac{i c \delta x}{2}}{2 \lambda + \delta \omega + (\lambda c^2 + \omega^2)^\frac{1}{2} \cosh \big( (4 (\lambda + \delta \omega) - c^2)^\frac{1}{2} x \big)} \times\\
& \times \bigg( \Big( 2 (\lambda c^2 + \omega^2)^\frac{1}{2} + c^2 - 2 \delta \omega \Big)^\frac{1}{2} \cosh \Big( \Big( (\lambda + \delta \omega) - \frac{c^2}{4} \Big)^\frac{1}{2} x \Big)\\
& + i \sign(c) \delta \Big( 2 (\lambda c^2 + \omega^2)^\frac{1}{2} - c^2 + 2 \delta \omega \Big)^\frac{1}{2} \sinh \Big( \Big( (\lambda + \delta \omega) - \frac{c^2}{4} \Big)^\frac{1}{2} x \Big) \bigg),
\end{align*}
and
$$\forall x \in \R, [V_{c, \omega}]_2(x) = \delta \Bigg( 1 - \frac{4 (\lambda + \delta \omega) - c^2}{2 \lambda + \delta \omega + (\lambda c^2 + \omega^2)^\frac{1}{2} \cosh \Big( (4 (\lambda + \delta \omega) - c^2)^\frac{1}{2} x \Big)} \Bigg),$$
with $\delta \in \{ \pm 1 \}$. Moreover, when $c \neq 0$, the map $\check{V}_{c, \omega}$ can be lifted as $\check{V}_{c, \omega} = |\check{V}_{c, \omega}| e^{i \varphi_{c, \omega}}$, with
$$\varphi_{c, \omega}(x) = \frac{c \delta x}{2} + \sign(c) \delta \arctan \bigg( \Big( \frac{2 (\lambda c^2 + \omega^2)^\frac{1}{2} - c^2 + 2 \delta \omega}{2 (\lambda c^2 + \omega^2)^\frac{1}{2} + c^2 - 2 \delta \omega} \Big)^\frac{1}{2} \tanh \Big( \Big( (\lambda + \delta \omega) - \frac{c^2}{4} \Big)^\frac{1}{2} x \Big) \bigg).$$
\end{thm}

\begin{rem}
Observe that the numbers $\delta = \pm 1$ in this statement give account of the limit when $x \to + \infty$ of the function $[V_{c, \omega}]_2$.
\end{rem}

\begin{proof}
The proof follows the approach developed in \cite{deLaire3}. By classical regularity theory (see e.g.~\cite{Helein0}), all the solutions $v = (v_1, v_2, v_3)$ to~\eqref{TW} in $\boE(\R)$ are smooth, and all their derivatives are bounded. Given a non-trivial solution $v$ in $\boE(\R)$, this implies that 
$$v_2(\pm \infty) := \lim_{x \to \pm \infty} v_2(x) \in \{ -1, 1 \}, \quad \text{and} \quad \lim_{x \to \pm \infty} |v'(x)| = 0.$$
Taking the complex scalar product of the first equation of~\eqref{TW} by $\check{v}'$, multiplying the second one by $v_2'$, and summing the resulting identities, we obtain
$$\big( |v'|^2 \big)' = - 2 \omega v_2' - 2 \lambda v_2 v_2',$$
due to the fact that $|v| = 1$. Integrating this expression from either $- \infty$ or $+ \infty$ gives
\begin{equation}
\label{dem:der-v}
 |v'|^2 = \lambda \big( 1 - v_2^2 \big) - 2 \omega \big( v_2 - v_2(\pm \infty) \big).
\end{equation}
In particular, we observe that
$$v_2(- \infty) = v_2(+ \infty),$$
except possibly when $\omega = 0$. As a consequence, we infer that the energy density
$$e(v) := \frac{1}{2} \big( |v'|^2 + \lambda |\check{v}|^2 \big)$$
is equal to
\begin{equation}
\label{dem-e}
e(v) = \lambda \big( 1 - v_2^2 \big) - \omega \big( v_2 - v_2(\pm \infty) \big).
\end{equation}
Taking the complex scalar product of the first equation of~\eqref{TW} by $i \check{v}$ and integrating the resulting expression leads to
\begin{equation}
\label{dem-mom}
\langle i \check{v}, \check{v}' \rangle_\C = c (v_2(\pm \infty) - v_2 \big),
\end{equation}
and we observe that $c$ is also equal to $0$ in case $v_2(+ \infty) \neq v_2(- \infty)$. Introducing~\eqref{dem-e} and~\eqref{dem-mom} in the second equation of~\eqref{TW}, we finally obtain the second-order differential equation for the function $v_2$
\begin{equation}
\label{dem:v2}
- v_2'' + \big( v_2 - v_2(\pm \infty) \big) \big( 2 \lambda v_2^2 + (2 \lambda v_2(\pm \infty) + 3 \omega) v_2 - c^2 + v_2(\pm \infty) \omega \big) = 0.
\end{equation}

At this stage, we split the analysis into two cases according to the values of $v_2(- \infty)$ and $v_2(+ \infty)$. If they are different, we know that $c = \omega = 0$, and~\eqref{dem:v2} reduces to
$$- v_2'' - 2 \lambda v_2 \big( 1 - v_2^2 \big) = 0.$$
Multiplying this equation by $v'_2$ and integrating, we get
$$- \big( v_2' \big)^2 - 2 \lambda v_2 ^2 + \lambda v_2^4 + \lambda = 0.$$
By the intermediate value theorem, we also know that $v_2$ vanishes. Up to a translation, we can assume that $v_2(0) = 0$. Hence, $v_2'(0)$ is equal to $\pm \lambda^{1/2}$. Coming back to~\eqref{dem:v2} and invoking the Cauchy-Lipshitz theorem, we conclude that there exist only two possible solutions depending on the value of $v_2'(0)$. Finally, we check that the functions
\begin{equation}
\label{dem-val-v20}
v_2(x) = \pm \tanh \big( \lambda^\frac{1}{2} x \big),
\end{equation}
are these two possible solutions. We next go back to the first equation in~\eqref{TW} and use~\eqref{dem-e} in order to write
\begin{equation}
\label{dem-eq-cv0}
- \check{v}'' + \lambda \big( 2 v_2^2 - 1 \big) \check{v} = 0.
\end{equation}
Again by~\eqref{dem-e}, we obtain
$$|\check{v}'(0)|^2 = \lambda \big( 1 - v_2(0)^2 \big) - v_2'(0)^2 = 0.$$
We are now reduced to invoke the Cauchy-Lipshitz theorem as before, and then to solve explicitly~\eqref{dem-eq-cv0}, so as to prove the existence of a complex number $\alpha$ such that
\begin{equation}
\label{dem-val-checkv0}
\check{v}(x) = \frac{\alpha}{\cosh(\lambda^\frac{1}{2} x)},
\end{equation}
for any $x \in \R$. Since $|\check{v}|^2 = 1 - v_2^2$, we infer that $|\alpha| = 1$, and up to a phase shift of the function $\check{v}$, we can assume that $\alpha = 1$. Statement $(i)$ in Theorem~\ref{thm:soliton} then follows from formulae~\eqref{dem-val-v20} and~\eqref{dem-val-checkv0}.

We now assume that $v_2(- \infty) = v_2(+ \infty) := v_2^\infty \in \{ - 1, 1 \}$. Multiplying~\eqref{dem:v2} by $v_2'$ and integrating, we have
\begin{equation}
\label{dem:derv2}
\big( v_2' \big)^2 = \big( v_2 - v_2^\infty \big)^2 \Big( \lambda \big( v_2 + v_2^\infty \big)^2 + 2 \omega \big( v_2 + v_2^\infty \big) - c^2 \Big),
\end{equation}
so that
\begin{equation}
\label{dem:condv2}
\lambda \big( v_2(x) + v_2^\infty \big)^2 + 2 \omega \big( v_2(x) + v_2^\infty \big) - c^2 \geq 0,
\end{equation}
for any $x \in \R$. Taking the limit $x \to \infty$, we obtain the necessary condition
$$c^2 \leq 4 \big( \lambda + \omega v_2^\infty \big),$$
for having a non-trivial solution, and we assume this condition to be fulfilled in the sequel.
 
At this stage, we also know that $v_2$ owns a global minimum when $v_2^\infty =1$, respectively maximum when $v_2^\infty = - 1$. This optimum is different from $\pm 1$, otherwise it follows from applying the Cauchy-Lipschitz theorem to~\eqref{dem:v2} that the solution $v$ is trivial. Up to a translation, we can assume that this optimum is attained at $x = 0$, so that $v_2'(0) = 0$, with $- 1 < v_2(0) < 1$. In view of~\eqref{dem:derv2}, this gives
$$\lambda \big( v_2(0) + v_2^\infty \big)^2 + 2 \omega \big( v_2(0) + v_2^\infty \big) - c^2 = 0.$$
Using~\eqref{dem:condv2}, we deduce that
\begin{equation}
\label{dem:valv20}
v_2(0) = \frac{1}{\lambda} \Big( v_2^\infty (\omega^2 + c^2 \lambda)^\frac{1}{2} - \omega - \lambda v_2^\infty \Big).
\end{equation}
In particular, the inequality $- 1 < v_2(0) < 1$ leads to the strongest necessary conditions
\begin{equation}
\label{dem:condv}
\text{or } \begin{cases}
0 < - \omega v_2^\infty < \lambda, \text{ and } c^2 < 4 \big( \lambda + \omega v_2^\infty \big),\\
\omega v_2^\infty \geq 0, \text{ and } 0 < c^2 < 4 \big( \lambda + \omega v_2^\infty \big),
\end{cases}
\end{equation}
which we assume in the sequel. In view of~\eqref{dem:valv20}, and since $v_2'(0) = 0$, applying the Cauchy-Lipschitz theorem to~\eqref{dem:v2} then provides the uniqueness, if existence, of the solution $v_2$. Note also that, as a further consequence of the Cauchy-Lipschitz theorem, the function $v_2$ is even. 

In order to construct this solution, we set $y = 1 - v_2^\infty v_2$. The function $y$ is even, positive and owns a global maximum at $x = 0$. Rewriting~\eqref{dem:derv2} in terms of the function $y$ and taking the square root of the resulting expression, we deduce that
\begin{equation}
\label{dem:eq-y}
y'(x) = - y(x) \big( \lambda y(x)^2 - 2 (\omega v_2^\infty + 2 \lambda) y(x) + 4 (\lambda + \omega v_2^\infty) - c^2 \big)^\frac{1}{2},
\end{equation}
for any $x \geq 0$. Observe here that the polynomial $P(X) = \lambda X^2 - 2 (\omega v_2^\infty + 2 \lambda) X + 4 (\lambda + \omega v_2^\infty) - c^2$ owns two distinct roots
$$X_+ := 2 + \frac{1}{\lambda} \Big( \sqrt{\omega^2 + \lambda c^2} + \omega v_2^\infty \Big) \geq 2 > X_- := 2 - \frac{1}{\lambda} \Big( \sqrt{\omega^2 + \lambda c^2} - \omega v_2^\infty \Big),$$
due to~\eqref{dem:condv}. Since $y(0) = X_-$ by~\eqref{dem:valv20} and since $y$ attains its global maximum at $x = 0$, it follows that the quantity $P(y(x))$ in the right-hand side of~\eqref{dem:eq-y} is positive for any $x \neq 0$. In view of~\eqref{dem:condv}, we are also allowed to write
$$\frac{\text{d}}{\text{d}y} \bigg( \arcosh \Big( \frac{4(\lambda + \omega v_2^\infty) - c^2 - (\omega v_2^\infty + 2 \lambda) y}{y \sqrt{\omega^2 + \lambda c^2}} \Big) \bigg) = - \frac{\big( 4 (\lambda + \omega v_2^\infty) - c^2 \big)^\frac{1}{2}}{y \sqrt{P(y)}},$$
for any $0 < y < X_-$. The existence of the solution $v_2$, as well as its expression in Statement $(ii)$ of Theorem~\ref{thm:soliton}, then follow from combining this formula with~\eqref{dem:eq-y}.

We next proceed with the first equation in~\eqref{TW}. Inserting~\eqref{dem-e} in this equation, we obtain the second-order linear differential equation
\begin{equation}
\label{eq:checkv}
- \check{v}'' + i c v_2 \check{v}' + \big( 2 \lambda v_2^2 - i c v_2' + 3 \omega v_2 - \lambda - 2 \omega v_2^\infty \big) \check{v} = 0.
\end{equation}
Since $|\check{v}|^2 = 1 - v_2^2$, the function $|\check{v}|^2$ owns a global maximum for $x = 0$, which is given by
$$|\check{v}(0)|^2 = \frac{2 (\lambda + \omega v_2^\infty) \big( (\omega^2 + c^2 \lambda)^\frac{1}{2} - \omega v_2^\infty \big) - c^2 \lambda}{\lambda^2} \neq 0,$$
in view of~\eqref{dem:valv20}. Using the invariance by phase shift of the function $\check{v}$, we can assume that
\begin{equation}
\label{IC-v}
\check{v}(0) = \frac{1}{\lambda} \Big( 2 (\lambda + \omega v_2^\infty) \big( (\omega^2 + c^2 \lambda)^\frac{1}{2} - \omega v_2^\infty \big) - c^2 \lambda \Big)^\frac{1}{2},
\end{equation}
By maximality at $x = 0$, we also know that
$$\langle \check{v}(0), \check{v}'(0) \rangle_\C = 0,$$
while~\eqref{dem-mom} provides
$$\langle i \check{v}(0), \check{v}'(0) \rangle_\C = c (v_2^\infty - v_2(0) \big).$$
Hence, we have
\begin{equation}
\label{IC-der-v}
\check{v}'(0) = i c \frac{v_2^\infty - v_2(0)}{|\check{v}(0)|^2} \check{v}(0) = i c v_2^\infty \Big( \frac{2 \lambda + \omega v_2^\infty - (\omega^2 + c^2 \lambda)^\frac{1}{2}}{(\omega^2 + c^2 \lambda)^\frac{1}{2} - \omega v_2^\infty} \Big)^\frac{1}{2}.
\end{equation}
In view of~\eqref{IC-v} and~\eqref{IC-der-v}, we deduce as before from the Cauchy-Lipschitz theorem that there exists at most one solution $\check{v}$ to~\eqref{eq:checkv}.

In order to conclude the proof, we are left with the construction of this solution. We split the analysis into two cases. When $c = 0$, equation~\eqref{eq:checkv} reduces to
 \begin{equation}
\label{eq:checkv0}
- \check{v}'' + \big( 2 \lambda v_2^2 + 3 \omega v_2 - \lambda - 2 \omega v_2^\infty \big) \check{v} = 0.
\end{equation}
Taking the real and imaginary parts of this equation, we infer that $v_1$ and $v_3$ also solve it. By~\eqref{IC-v} and~\eqref{IC-der-v}, we have $v_3(0) = v_3'(0) = 0$. By the Cauchy-Lipschitz theorem once again, the function $v_3$ identically vanishes, and we obtain
$$v_1(x)^2 = 1 - v_2(x)^2 = - \frac{8 v_2^\infty \omega (\lambda + v_2^\infty \omega) \big( 1 + \cosh(2 (\lambda + v_2^\infty \omega)^\frac{1}{2} x) \big)}{\big( 2 \lambda - v_2^\infty \omega \big( \cosh(2 (\lambda + v_2^\infty \omega)^\frac{1}{2} x) - 1 \big) \big)^2},$$
for any $x \in \R$. Here, we have used the property that $|\omega| = - v_2^\infty \omega$ due to conditions~\eqref{dem:condv}. Since $v_1'(0) = 0$ and
$$v_1(0) = - \frac{2 \sqrt{2}}{\lambda} \big( \omega v_2^\infty \lambda + \omega^2 \big)^\frac{1}{2},$$
 by~\eqref{IC-v} and~\eqref{IC-der-v}, we infer that
 $$\check{v}(x) = v_1(x) = \frac{4 \big( - v_2^\infty \omega \lambda - \omega^2 \big)^\frac{1}{2} \cosh \big( (\lambda + v_2^\infty \omega)^\frac{1}{2} x \big)}{2 \lambda - v_2^\infty \omega \big( \cosh(2 (\lambda + v_2^\infty \omega)^\frac{1}{2} x) - 1 \big)},$$
 is the desired solution to~\eqref{eq:checkv}.
 
 We finally turn to the case $c \neq 0$. We now compute
 $$|\check{v}(x)|^2 = \frac{\big( 4 (\lambda + v_2^\infty \omega) - c^2 \big) \Big( 2 (\lambda c^2 + \omega^2)^\frac{1}{2} \cosh \big( (4 (\lambda + v_2^\infty \omega) - c^2)^\frac{1}{2} x \big) + c^2 - 2 v_2^\infty \omega \Big)}{\Big( 2 \lambda + v_2^\infty \omega + (\lambda c^2 + \omega^2)^\frac{1}{2} \cosh \big( (4 (\lambda + v_2^\infty \omega) - c^2)^\frac{1}{2} x \big) \Big)^2},$$
for any $x \in \R$. In this case, we observe that the function $\check{v}$ does not vanish on $\R$, so that we can lift it as $\check{v} = |\check{v}| e^{i \varphi}$ with a smooth phase function $\varphi$. Coming back to~\eqref{dem-mom}, we compute 
$$\varphi'(x) = \frac{c}{v_2^\infty + v_2(x)} = \frac{c v_2^\infty}{2} \bigg( 1 + \frac{4 (\lambda + v_2^\infty \omega) - c^2}{2 (\lambda c^2 + \omega^2)^\frac{1}{2} \cosh \big( (4 (\lambda + v_2^\infty \omega) - c^2)^\frac{1}{2} x \big) + c^2 - 2 v_2^\infty \omega} \bigg).$$
Since $\check{v}(0) > 0$ by~\eqref{IC-v}, we can also assume that $\varphi(0) = 0$. Checking that
$$\frac{\text{d}}{\text{d}y} \bigg( \frac{2 a}{b (1 - d^2)^\frac{1}{2}} \arctan \Big( \Big( \frac{1 - d}{1 + d} \Big)^\frac{1}{2} \tanh \Big( \frac{a y}{2} \Big) \Big) \bigg) = \frac{a^2}{b \big( \cosh(a y) + d \big)},$$
for any $y \in \R$, and any coefficients $a > 0$, $b > 0$ and $- 1 < d < 1$, we conclude that the phase function $\varphi$ is given by the formula in Statement $(ii)$ of Theorem~\ref{thm:soliton}. Finally, the formula for the map $\check{v}$ follows from the ones for $|\check{v}|$ and $\varphi$ using the identities
$$\forall x \in \R, \cos(\arctan(x)) = \frac{1}{(1 + x^2)^\frac{1}{2}}, \quad \text{and} \quad \sin(\arctan(x)) = \frac{x}{(1 + x^2)^\frac{1}{2}}.$$
This concludes the proof of Theorem~\ref{thm:soliton}.
\end{proof}

We now go back to the cubic Schr\"odinger regime of the Landau-Lifshitz equation. In view of the classification in Theorem~\ref{thm:soliton}, the solitons $\Psi_{c, \omega}$ of~\eqref{CS} can be obtained in this regime as the limit of the solitons $m_{c_\varepsilon, \omega_\varepsilon}$ of~\eqref{LL} for the choice of parameters
\begin{equation}
\label{def:param-eps}
c_\varepsilon = c, \quad \delta_\varepsilon = 1, \quad \lambda_\varepsilon = \frac{1}{\varepsilon}, \quad \text{and} \quad \omega_\varepsilon = \omega - \frac{1}{\varepsilon}.
\end{equation}
Indeed, fix $\omega > 0$ and $c \geq 0$, with $4 \omega > c^2$, so that the soliton $\Psi_{c, \omega}$ is well-defined. The assumptions $0 < - \omega_\varepsilon \delta_\varepsilon$ and $c_\varepsilon^2 < 4 \big( \lambda_\varepsilon + \omega_\varepsilon \delta_\varepsilon \big)$ then hold for $\varepsilon$ small enough. By Theorem~\ref{thm:soliton}, there exist non-trivial solutions $V_{c_\varepsilon, \omega_\varepsilon}$ to~\eqref{TW} with
\begin{equation}
\label{eq:V-eps}
\begin{split}
\check{V}_{c_\varepsilon, \omega_\varepsilon}(x) = & \frac{\varepsilon^\frac{1}{2} \, (4 \omega - c^2)^\frac{1}{2} \, e^\frac{i c x}{2}}{1 + \omega \varepsilon + (1 + (c^2 - 2 \omega) \varepsilon + \omega^2 \varepsilon^2)^\frac{1}{2} \cosh \big( (4 \omega - c^2)^\frac{1}{2} x \big)} \times\\
& \times \bigg( \big( 2 (1 + (c^2 - 2 \omega) \varepsilon + \omega^2 \varepsilon^2)^\frac{1}{2} + 2 + (c^2 - 2 \omega) \varepsilon \big)^\frac{1}{2} \cosh \Big( \Big( \omega - \frac{c^2}{4} \Big)^\frac{1}{2} x \Big)\\
& + i \sign(c) \big( 2 (1 + (c^2 - 2 \omega) \varepsilon + \omega^2 \varepsilon^2)^\frac{1}{2} - 2 + (2 \omega - c^2) \varepsilon \big)^\frac{1}{2} \sinh \Big( \Big( \omega - \frac{c^2}{4} \Big)^\frac{1}{2} x \Big) \bigg).
\end{split}
\end{equation}
Coming back to the scaling in~\eqref{def:Psi-eps}, we observe that the corresponding function
\begin{equation}
\label{def:U-eps}
\Upsilon_\varepsilon(x, t) := \varepsilon^{- \frac{1}{2}} \check{m}_{c_\varepsilon, \omega_\varepsilon}(x, t) e^\frac{i t}{\varepsilon} = \varepsilon^{- \frac{1}{2}} \check{V}_{c_\varepsilon, \omega_\varepsilon}(x - c t) e^{i \omega t},
\end{equation}
satisfies
$$\Upsilon_\varepsilon(x, t) \to \Psi_{c, \omega}(x, t),$$
as $\varepsilon \to 0$. Moreover, we can control the difference between the functions $\Upsilon_\varepsilon$ and $\Psi_{c, \omega}$ by a factor of order $\varepsilon$ as in Theorem~\ref{thm:conv-CLS}.

\begin{prop}
\label{prop:diff-solitons}
Let $k \in \N$, $c \geq 0$ and $\omega > 0$, with $4 \omega > c^2$. Given any number $0 < \varepsilon < 1/\omega$, consider the function $\Upsilon_\varepsilon$ defined by~\eqref{def:U-eps}, with the choice of parameters $c_\varepsilon$, $\delta_\varepsilon$, $\lambda_\varepsilon$ and $\omega_\varepsilon$ given by~\eqref{def:param-eps}. There exists a positive number $C_k$, depending only on $k$, $c$ and $\omega$, such that
$$\big\| \Upsilon_\varepsilon(\cdot, t) - \Psi_{c, \omega}(\cdot, t) \big\|_{H^k} \underset{\varepsilon \to 0}{\sim} C_k \varepsilon,$$
for any $t \in \R$.
\end{prop}

\begin{proof}
Given any real number $t$, we infer from~\eqref{sol:NLS} and~\eqref{def:U-eps} that
\begin{equation}
\label{caroll}
\big\| \Upsilon_\varepsilon(\cdot, t) - \Psi_{c, \omega}(\cdot, t) \big\|_{H^k} = \big\| \varepsilon^{- \frac{1}{2}} \check{V}_{c_\varepsilon, \omega_\varepsilon} - U_{c, \omega} \big\|_{H^k},
\end{equation}
where we have set
$$U_{c, \omega}(x) = \frac{2 \, \alpha \, e^{i \frac{c x}{2}}}{\cosh \big( \alpha x \big)},$$
with $\alpha := (\omega - c^2/4)^{1/2}$. In view of~\eqref{eq:V-eps}, we have
\begin{equation}
\label{lafaurie}
\varepsilon^{- \frac{1}{2}} \check{V}_{c_\varepsilon, \omega_\varepsilon}(x) - U_{c, \omega}(x) = \varepsilon W_{c, \omega}(x) + \varepsilon^2 R_\varepsilon(x),
\end{equation}
where the first-order term $W_{c, \omega}$ is equal to
$$W_{c, \omega}(x) := \frac{\alpha \, e^{i \frac{c x}{2}}}{4 \cosh(\alpha x)^2} \Big( (4 \alpha^2 - c^2) \cosh(\alpha x) - \frac{8 \alpha^2}{\cosh(\alpha x)} + 4 i c \alpha \sinh(\alpha x) \Big),$$
whereas the remainder term $R_\varepsilon$ is given by 
\begin{align*}
R_\varepsilon(x) := \frac{\alpha \, e^{i \frac{c x}{2}}}{2 \cosh(\alpha x)^4} \bigg( 2 \big( \Gamma_\varepsilon - 2 N_\varepsilon - \nu_\varepsilon \gamma_\varepsilon \big) \cosh(\alpha x)^3 & - \big( \gamma_\varepsilon (\omega - \nu_\varepsilon) - 2 N_\varepsilon \big) \cosh(\alpha x)\\
+ i \sign(c) \sinh(\alpha x) \big( 2 (K_\varepsilon - \nu_\varepsilon \kappa_\varepsilon) \cosh(\alpha x)^2 & - (\omega - \nu_\varepsilon) \kappa_\varepsilon) \big)\\
+ \frac{\big( \omega - \nu_\varepsilon + 2 \nu_\varepsilon \cosh(\alpha x)^2 \big)^2}{2 (1 + \varepsilon \nu_\varepsilon) \cosh(\alpha x)^2 + \varepsilon (\omega - \nu_\varepsilon)} \big( (2 + \varepsilon & \gamma_\varepsilon) \cosh(\alpha x) + i \varepsilon \sign(c) \kappa_\varepsilon \sinh(\alpha x) \big) \bigg).
\end{align*}
In this expression, we have set
$$\gamma_\varepsilon := \frac{c^2 - 2 \omega}{2} + \varepsilon \Gamma_\varepsilon := \frac{1}{\varepsilon} \Big(\big( 2 (1 + (c^2 - 2 \omega) \varepsilon + \omega^2 \varepsilon^2)^\frac{1}{2} + 2 + (c^2 - 2 \omega) \varepsilon \big)^\frac{1}{2} - 2 \Big),$$
$$\kappa_\varepsilon := \frac{|c| (4 \omega - c^2)^\frac{1}{2}}{2} + \varepsilon K_\varepsilon := \frac{1}{\varepsilon} \Big( 2 (1 + (c^2 - 2 \omega) \varepsilon + \omega^2 \varepsilon^2)^\frac{1}{2} - 2 + (2 \omega - c^2) \varepsilon \Big)^\frac{1}{2},$$
and
$$\nu_\varepsilon := \frac{c^2 - 2 \omega}{2} + \varepsilon N_\varepsilon := \frac{1}{\varepsilon} \Big( \big( 1 + (c^2 - 2 \omega) \varepsilon + \omega^2 \varepsilon^2 \big)^\frac{1}{2} - 1 \Big).$$
Since
$$\Gamma_\varepsilon \to \frac{4 \omega c^2 - c^4 - 2 \omega^2}{8}, \quad K_\varepsilon \to \frac{|c| (2 \omega - c^2) (4 \omega - c^2)^\frac{1}{2}}{8}, \quad \text{and} \quad N_\varepsilon \to \frac{c^2 (4 \omega -c^2)}{8},$$
as $\varepsilon \to 0$, it follows from the smoothness and the exponential decay of the remainder term $R_\varepsilon$ that its $H^k$-norm remains bounded as $\varepsilon \to 0$. In view of~\eqref{lafaurie}, we deduce that
$$\big\| \varepsilon^{- \frac{1}{2}} \check{V}_{c_\varepsilon, \omega_\varepsilon} - U_{c, \omega} \big\|_{H^k} \underset{\varepsilon \to 0}{\sim} \varepsilon \| W_{c, \omega} \|_{H^k}.$$
It is then enough to set $C_k := \| W_{c, \omega} \|_{H^k}$ and to use~\eqref{caroll} so as to complete the proof of Proposition~\ref{prop:diff-solitons}.
\end{proof}

\begin{merci}
The authors acknowledge support from the project ``Dispersive and random waves'' (ANR-18-CE40-0020-01) of the Agence Nationale de la Recherche, and from the grant ``Qualitative study of nonlinear dispersive equations'' (Dispeq) of the European Research Council. Andr\'e~de~Laire was partially supported by the Labex CEMPI (ANR-11-LABX-0007-01) and the MathAmSud program.
\end{merci}

\bibliographystyle{plain}
\bibliography{Bibliogr}

\end{document}